\documentclass[11pt,times]{elsarticle}

\usepackage[utf8]{inputenc}
\usepackage[english]{babel}
\usepackage{easyReview}
\usepackage{mathrsfs}
\usepackage{amsmath}
\usepackage{amsfonts}
\usepackage{amssymb}
\usepackage{amsthm}
\usepackage{mathtools}
\usepackage{bbm}
\usepackage[a4paper, total={6.25in, 9in}]{geometry}

\usepackage{braket}
\usepackage{enumerate}
\usepackage{esint}

\usepackage{color}   
\usepackage{hyperref}
\hypersetup{
    colorlinks=true, 
    linktoc=all,     
    linkcolor=blue,  
}
\usepackage{graphicx} 

\mathtoolsset{
	showonlyrefs
}
\numberwithin{equation}{section}

\theoremstyle{definition} \newtheorem{deff}{Definition} [section]
\theoremstyle{definition} 
\theoremstyle{definition} \newtheorem{thm} [deff]{Theorem}
\theoremstyle{definition} \newtheorem{lm} [deff]{Lemma}
\theoremstyle{remark} 
\theoremstyle{remark} \newtheorem{rmk}[deff]{Remark} 
\theoremstyle{plain}  

\newtheoremstyle{claim}
  {}
  {}
  {\itshape}
  {}
  {}
  {.}
  { }
  {}

\theoremstyle{claim}

\renewenvironment{rmk}{\begin{oldrmk}}{\fr\end{oldrmk}\\ \par}

\newcommand{\R}{\mathbb{R}}

\newcommand{\restr}[1]{\lower3pt\hbox{$|_{#1}$}}

\renewcommand{\d}{{\rm d}}
		
\newcommand{\nchi}{{\raise.3ex\hbox{$\chi$}}}

\newcommand{\F}{{\mathcal F}}

\newcommand{\fr}{\penalty-20\null\hfill$\blacksquare$}         

\begin{document}

\begin{frontmatter}

\title{Elastic bodies with kinematic constraints on many small regions}

\author[aff1]{Andrea Braides\texorpdfstring{\corref{cor1}}{}}
\ead{abraides@sissa.it}

\author[aff1]{Giovanni Noselli}

\author[aff1]{Simone Vincini}

\address[aff1]{SISSA\,--\,International School for Advanced Studies, Mathematics Area, via Bonomea 265, 34136 Trieste, Italy}

\begin{abstract}
	We study the equilibrium of hyperelastic solids subjected to kinematic constraints on many small regions, which we call perforations.
    Such constraints on the displacement $u$ are given in the quite general form $u(x) \in F_x$, where $F_x$ is a closed set, and thus comprise confinement conditions, unilateral constraints, controlled displacement conditions, etc., both in the bulk and on the boundary of the body. The regions in which such conditions are active are assumed to be so small that they do not produce an overall rigid constraint, but still large enough so as to produce a non-trivial effect on the behaviour of the body. Mathematically, this is translated in an asymptotic analysis by introducing two small parameters: $\varepsilon$, describing the distance between the elements of the perforation, and $\delta$, the size of the element of the perforation.  We find the critical scale at which the effect of the perforation is non-trivial and express it in terms of a $\Gamma$-limit in which the constraints are relaxed so that, in their place, a penalization term appears in the form of an integral of a function $\varphi(x,u)$. This function is determined by a blow-up procedure close to the perforation and depends on the shape of the perforation, the constraint $F_x$, and the asymptotic behaviour at infinity of the strain energy density $\sigma$. We give a concise  proof of the mathematical result and numerical studies for some simple yet meaningful geometries.
\end{abstract}



\end{frontmatter}
\bigskip
\hfill{\em Dedicated to Kaushik Bhattacharya on the occasion of his sixtieth birthday}


%
\section{Introduction}

The title of this work echoes the one of a paper by the first author and the dedicatee of this work and this Special Issue \cite{BK}. In that paper, the behaviour of an elastic thin film with many small cracks was analyzed, showing a limit behaviour different from the one usually obtained in dimension-reduction theories, with an additional contribution due to the overall deformation of the cracks. Similarly, here we consider a problem in which an increasingly complex geometry of the sets in which kinematic constrains are imposed leads to a non-trivial additional energy contribution. We have in mind a classical problem in Mechanics, which consists in determining the deformation of an elastic body under prescribed displacements at its boundary. Weak solutions to this problem can be achieved via minimization of the strain-energy functional among the admissible displacement fields~\cite{gurtin_1981_a, gurtin_1981_b}.
This setting can be enriched upon introducing kinematic constraints on either periodic or locally periodic point subsets of the reference configuration, which, in line with the existing mathematical literature, are referred to as \emph{perforations}. 

In order to start with a naive example, consider the action of fixing a sheet of wallpaper to a wooden ceiling with a set of drawing pins. If the pins are few, the sheet can still deform away from them but senses their presence in their vicinity. If the material is rigid and we use many pins, the sheet will adhere to the ceiling independently of the size of the pins. If instead we use a sheet made of a less rigid material (think of a latex sheet) and an array of very small pins sufficiently sparsely distributed, their effect will be sensed by the material only close to the pins, and the sheet will be able to detach from the ceiling, kept attached only by small filaments. The overall effect of the pins will be a penalization of the averaged distance from the zero-displacement state, even though at the pin sites a zero-displacement (Dirichlet) condition is always satisfied. In the case of extremely small pins, the filaments, even though still present, carry a negligible amount of energy and the sheet may freely detach. In a variational setting, problems of this type have been widely considered, in terms of the asymptotic description of the behaviour of perforated domains subjected to Dirichlet boundary conditions, see \cite{March-Krus,Cio-Mur}.

\begin{figure}[t]
    \centering
\includegraphics[width=1.0\linewidth]{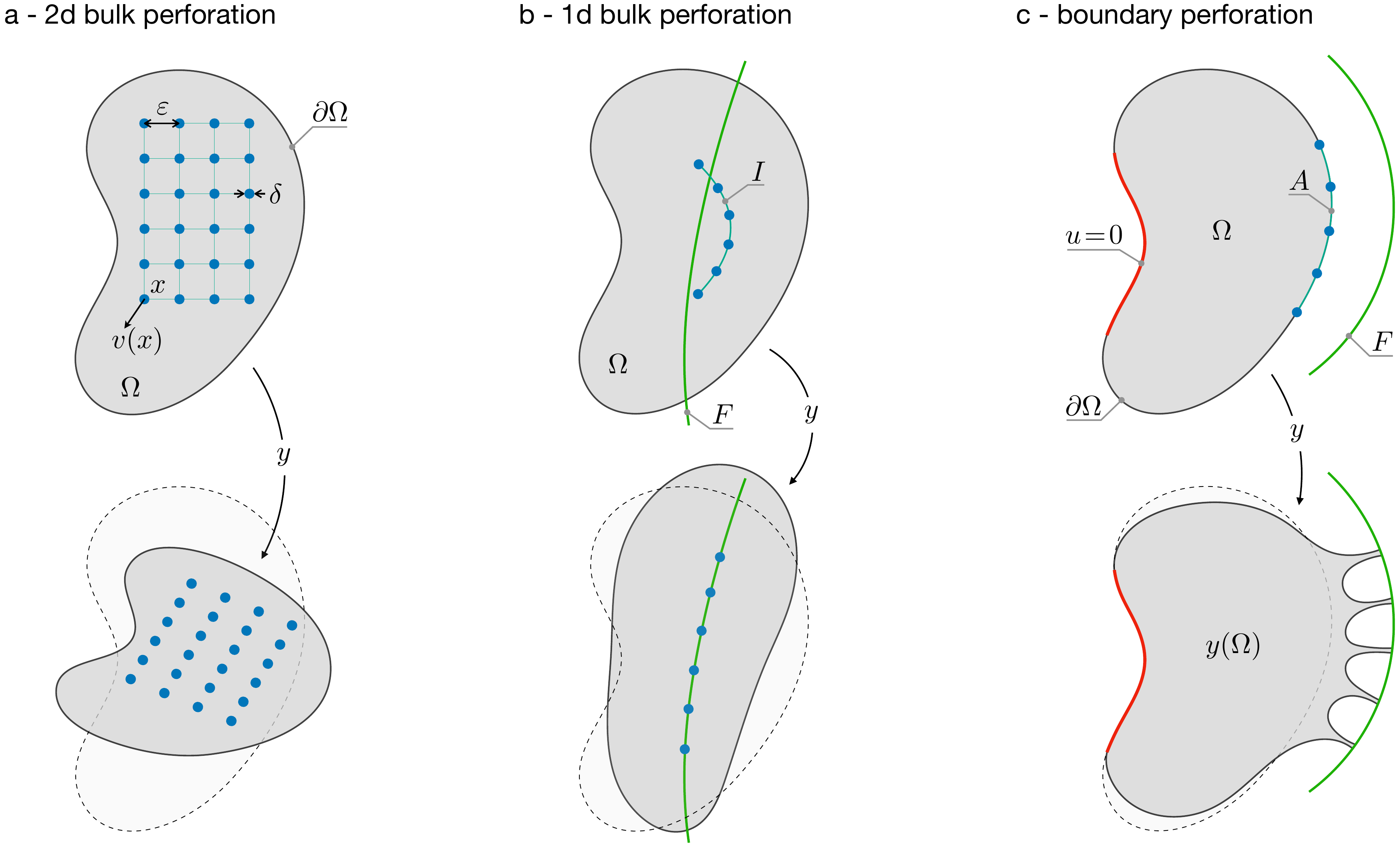}
    \caption{Examples of perforations on a reference set $\Omega \subset \R^2$ and relative constraints. (a) A bulk perforation in the reference set is subjected to the constraint $F_x = \{v(x)\}$, such that the function $v:\mathcal{P} \to \R^2$ imposes the deformation on the perforation. In particular, the figure reports the case of an affine deformation. (b) A lower dimensional perforation on the subset $I \subset \Omega$ is constrained to lay on the target set $F$ (green curve). (c) A boundary perforation on $A \subset \partial\Omega$ is forced on the target set $F$ (green curve). For all cases, the  reference (deformed) configuration is reported on the first (second) row of the figure.}
    \label{fig:intro}
\end{figure}

Keeping our naive example in mind, we now want to extend our analysis either to pins that, for example, allow the wallpaper to move in some directions but not others (unilateral constraints) or loose pins that allow a certain detachment from the pinning site (confinement constraint).
To this end, we extend the class of amenable constraints by requiring the deformation to belong to a general class of closed sets parameterized by the sites of the perforation. More explicitly, denoting by  $y\colon\Omega\to\mathbb R^m$ the deformation of the reference set $\Omega \subset \mathbb R^d$, 
we study the behaviour of hyperelastic energies subjected to vector constraints of the form $y(x)\in F_x$, where we identify by (the material point) $x$ the sites of the perforation $\mathcal P$ in the reference configuration and $F_x\subset\mathbb{R}^m$ is a closed set.
A first example we have in mind is that of $F_x= \{\varv(x)\}$, in which the function $\varv\colon\mathcal{P}\to\mathbb R^m$ prescribes the deformation at the perforation, see Fig.~\ref{fig:intro}a. Another relevant example comes from $F_x=F$, where $F$ is a fixed closed set on which the perforation is constrained, such as ${\mathbb S}^{d-1}$. In particular, such constraints may be imposed on a lower-dimensional set $I \subset \Omega$, see Fig.~\ref{fig:intro}b.
Furthermore, the perforation may lay on a subset $A$ of the boundary $\partial\Omega$ and act in this case as a constraint of \lq loose confinement' for the hyperelastic body. As a prototype, we  may consider the case of $F_x=F$ a target set, as shown in see Fig.~\ref{fig:intro}c, or that of $F_x$ a cone with axis singled out by the normal to $\partial\Omega$. Clearly, these figures are only cartoons to give an idea of a much broader mathematical setting. In Fig.~\ref{fig:intro}c we have the possibility to picture the deformation essentially as a graph, highlighting the pinning effect of the perforation, which is not easy to picture in a cartoon, but always present, for a perforation in the bulk.

For the class of vector constraints introduced above, we aim at describing the averaged behaviour of the minimizers of hyperelastic energies under suitable assumptions on the strain-energy density $\sigma:\R^{m\times d}\to\R$ and on the closed sets $F_x$, as later detailed. For this purpose, we assume that the period of the perforation is characterized by a small parameter $\varepsilon$. Moreover, since constraints at points are in general meaningless, we impose them on small sets, whose size is controlled by introducing a second small parameter $\delta$.
The overall behaviour of the hyperelastic energy can then be analyzed by first computing the critical scale of $\delta$ for which it is not trivial, and then describing the asymptotic behaviour of minimum problems at this scaling. It is by now common in this setting to express this limit behaviour, using the terminology of $\Gamma$-convergence, by computing the $\Gamma$-limit at this critical scale, which is just a compact way to express the convergence of minimum values and minimizers  (see \cite{GCB,Bha-Ja}).

We explore different mechanical scenarios and subdivide our results in the relevant mathematical statements. However, we anticipate a remarkable feature of the limit functional that is common to all the treated cases. Besides the classical strain-energy term, the $\Gamma$-limit governing the averaged behaviour of the hyperelastic energy is shown to comprise an integral term encoding the effect of the relaxed constraints on the perforation.

The mathematical setting we introduce significantly extends previous results on hyperelastic bodies subjected to constraints on perforations. While the goal of our study is to develop such mathematical setting, we foresee the applicability of our results to a wide range of mechanical contexts. For instance, they may be exploited to study the equilibrium configuration of membranes and tensile structures supported by perforations located at the boundary, in the interior, or both. This is a context that we explore through numerical simulations on simple test cases for illustrative purposes. Another field of application may be that concerning the shape morphing of elastic bodies from an internal or external control~\cite{ortigosa_2021, andrini_2022}, encoded by the specific constraints imposed on the perforation or that regarding the clothing of surfaces by elastic, micro-architected membranes~\cite{mina_2018, andrini_2024}.

The manuscript is organized as follows. Section 2 is devoted to the presentation of the mathematical setting, the working assumptions, and to the mathematical results, which we subdivide in different statements. Section 3 illustrates the applicability of the obtained results through the numerical solution of prototypical examples, whereas Section 4 closes the manuscript by reporting about future perspectives.

\section{Mathematical setting of the problem and main results}\label{sec:mat_results}
In this section we gather the analytical results of the paper. The content is organized so that a reader not wishing to concentrate on the mathematical statements may directly skip theorems and proofs, as their mechanical significance is described in the rest of the text. 
\smallskip

We start by fixing some miscellaneous notation used in the text. The strictly positive natural numbers $d$ and $m$ will be the dimension of the reference and target spaces, respectively. They can be arbitrary, even if we will have the physical situations in mind.
The space of real-valued $m\times d$ matrices is denoted by $\mathbb R^{m\times d}$. The euclidean norm of a point $x\in\mathbb R^k$ is denoted by $|x|$, of a matrix $\xi$ by $\|\xi\|$. A ball of center $x$ and radius $r$ is denoted by $B(x,r)$. The unit spherical surface in $\mathbb R^d$ is denoted by $\mathbb S^{d-1}$. The characteristic function of a set $A\subset \mathbb R^d$ is denoted by 
$\mathbbm{1}_A$. The symbol $\|\phi\|_{L^q}$ denotes the norm in the corresponding $L^q$ space.

\smallskip
We consider general hyperelastic energies 
defined on functions $u$ subjected to vector constraints on either periodic or locally periodic point subsets, which we call `perforations' by coherence with the existing mathematical literature. In our context it would be more correct to call such point sets `pinning sites' or `constraint sites'.
Since we are interested in an averaged behaviour of minimizers of such energies, we suppose that the period of the perforation is parameterized by a small parameter $\varepsilon$. Moreover, since point values of the function $u$ are in general meaningless, such constraints must indeed be imposed on small sets, whose size is parameterized by introducing a second small parameter $\delta$. In analytical terms we will consider, given $\Omega$ an open set in $\mathbb R^d$ acting as a reference configuration for $u$, a fixed hyperelastic energy
$$
\int_\Omega \sigma(\nabla u)\d x,
$$
 with domain, depending on $\varepsilon$ and $\delta$, the set of functions satisfying a condition $u(x)\in F_x$ for $x\in K_{\varepsilon,\delta}$, where $K_{\varepsilon,\delta}$ is a $\varepsilon$-periodic array of sets of size $\delta$ (specified further on) and $F_x$ is an $x$-dependent closed set.

The overall behaviour of minimum problems involving such energies can then be analyzed by first computing the critical scaling $\delta=\delta(\varepsilon)$ for which this behaviour is not trivial as $\varepsilon$ and $\delta$ tend to $0$, and then, using the terminology of $\Gamma$-convergence, by computing the $\Gamma$-limit at this critical scale. We recall that $\Gamma$-convergence is simply a compact way to express the fact that minimum problems depending on small parameters converge (in the sense that minimum values and minimizers convergence) to the corresponding problem  for the $\Gamma$-limit. Moreover, by the topological nature of its definition, $\Gamma$-convergence is stable by addition of continuous perturbations, so that in our case it automatically implies the convergence also of problems with forcing terms. Finally, for our functionals $\Gamma$-convergence is also compatible with the addition of boundary conditions. These two properties are very handful since they imply that the computation of one single $\Gamma$-limit allows the description of a whole family of problems with arbitrary forcing terms and boundary conditions (see \cite{GCB}). 

In the case when ($u$ is a scalar and) the constraint on the perforations is simply $u=0$, that is, $F_x$ is the set $\{0\}$ independently of $x$, the just described limit problem is classical and regards the behaviour of perforated domains subjected to Dirichlet boundary conditions on a diffuse set (of vanishing measure). The prototypical case is when $\sigma(\xi)=\|\xi\|^2$ in dimension $d=3$, for which $\delta=\varepsilon^3$ and the limit is $$\int_\Omega\|\nabla u\|^2\d x+C \int_\Omega |u|^2\d x,$$
with $C$ an explicit constant depending on the details of the perforation
(see \cite{March-Krus,Cio-Mur}; for an approach in the non-convex vector case we refer to \cite{AnBr02}). The additional term describes the limit effect of the diffuse Dirichlet condition. The scale $\delta=\varepsilon^3$ is critical in the sense that if $\delta<\!<\varepsilon^3$ then the effect of the perforation vanishes as $\varepsilon\to 0$ and in the limit we only obtain the first integral (that is, $C=0$ in the expression above), while if $\delta>\!>\varepsilon^3$ then the effect of the perforation becomes very strong as $\varepsilon \to 0$, and the only function with finite limit energy is $u=0$ (that is, in a sense $C=+\infty$  in the expression above).

For energies describing the behaviour of elastic bodies, more general constraints than a simple fixed Dirichlet condition can be of interest.
In this paper we study energies defined on vector functions $u$ and subjected to general constraints $u\in F_x$. In this notation, $x$ is a parameter for the sites of the perforation, and $F_x\subset \mathbb R^m$ is a closed set. As a model case we have in mind either that $F_x= \{\varv(x)\}$, where $\varv$ is a function describing the deformation of the perforation, or $F_x=F$, a fixed closed set such as $\mathbb S^{d-1}$,
which imposes some constraint on the displacement $u(x)$ at the perforation. Furthermore, we may also consider perforation lying at the boundary of $\Omega$, which act as a `loose confinement' constraint. In this case, we may take as prototypes $F_x=F$, where $F$ is a target set, or $F_x$ a cone with axis the direction of the inner normal to $\Omega$. 
We now describe more analytically our assumptions. 
\bigskip

 \noindent
 {\bf The energy density.} We assume that $\sigma\colon\R^{m\times d}\to\R$ is a general hyperelastic energy density satisfying the following assumptions:

 \smallskip
  {1.} $\sigma$ is {\em quasiconvex};
 
   {2.}   there exists $p\in (0,d)$ such that  $\sigma$ has {\em $p$-growth}; that is, there exist $c_1, c_2>0$ such that 
        \begin{equation}
            c_1 \|\xi\|^p-\frac{1}{c_1}\leq \sigma(\xi)\leq c_2\left(1+\|\xi\|^p\right) \quad \hbox{ for all $\xi\in \R^{m\times d}$;} 
        \end{equation}

          {3.} there exists the {\em recession function of $\sigma$}; that is, there exists  $\sigma_0\colon\R^{m\times d}\to \R$ \footnote{Compatible with the $\sigma_\varepsilon$ notation used in the following.} and the limit
        \begin{equation}\label{defsigma0}\lim_{t\to\infty}\frac{\sigma(t\xi)}{t^p}=\sigma_0(\xi) \quad \hbox{ for all $\xi\in \R^{m\times d}$.}
        \end{equation}

We will assume a quantitative version of Condition 3; namely, that there exists $\alpha>\frac{p}{d}$ and a constant $C>0$ such that
\begin{equation}\label{quantitative}
\bigg|\frac{\sigma(t\xi)}{t^p}-\sigma_0(\xi)\bigg|\leq \frac{C}{t^\alpha}(1+\|\xi\|^p).
\end{equation}
This is a mild technical condition satisfied by all reasonable energy densities admitting a recession function.

\smallskip
We note that the assumption that $\sigma$ is quasiconvex can be dropped, up to a preliminary relaxation argument, in which the integrands are replaced by their quasiconvex envelopes (see \cite{Bhattacharya-book,BrDe98}).  Note moreover that the function $\sigma_0$ is positively homogeneous of degree $p$ by definition; that is, $\sigma_0(t\xi)=t^p\sigma_0(\xi)$ if $t>0$. The existence of such a $\sigma_0$ holds for sequences of diverging $t$; the assumption that is independent of the subsequence is a uniformity assumption on the $p$-homogeneous behaviour for large values of $|\xi|$. In particular $\sigma_0=\sigma$ if $\sigma$ is positively homogeneous of degree $p$.
\bigskip

\noindent
 {\bf Perforations.} We consider a reference {\em perforation shape $K$}, which we assume to be a closed set. In this paper $K$ is either a regular closed set (that is, it is the closure of its interior), or a regular closed subset of a sufficiently smooth hypersurface. The {\em perforation set}, or simply {\em perforation}, will be obtained as a union of scaled copied of $K$, or small deformations of scaled copied of $K$, which will be specified in the different results.

\bigskip
\noindent
 {\bf Constraints.} We introduce an $x$-parameterized family of {\em constraint sets $F_x$}. The constraints on a function $u$ at scale $\varepsilon$ will be expressed in terms of an inclusion $u(z)\in F_{x_\varepsilon(z)}$ for $z$ in the perforation.  The value $x_\varepsilon(z)$ will be constant on each element of the perforation, so that this condition expresses the same constraint for the displacement on each element of the perforation. 
 Even though at fixed $\varepsilon$ only a discrete family of sets $F_x$ is taken into account, we find it convenient to define $F_x$ for all $x$ rather than introducing $\varepsilon$-dependent constraints for an easier description of the asymptotic behaviour.
 
 In order to have a limit behaviour as $\varepsilon$ tends to $0$ we make some assumptions on the {\em constraint set $F_x$}. We assume that the set function $x\mapsto F_x\in \mathcal P(\R^m)$ is such that the following conditions hold:

\begin{enumerate}
  \item the set $F_x$ is closed for all $x\in \overline{\Omega}$;
  \item there exist $M>0$ and a Lipschitz function $s\colon\overline{\Omega} \to \R^m$ such that $s(x)\in F_x\cap B(0, M)$ for all $x\in \overline{\Omega}$. Furthermore, there exists $C$ such that for all $j>2M$ there exists a $C$-Lipschitz function $\psi_j$ such that $\psi_j(x, u)=u$ for $|u|\leq j$, $\psi_j(x, u)=s(x)$ for $|u|\geq 2j$, $\psi_j(x, u)\in F_x$ for every $w\in F_x$, with $s$ as above;
  \item there exist a family of bi-Lipschitz diffeomorphisms $(\Phi_{x_1, x_2})_{x_1, x_2 }$ such that  $\Phi_{x_1, x_2}(F_{x_1})=F_{x_2}$,  $\Phi_{x, x}=id$, and:

   \hspace{15pt} i) for all $L>0$ there exists $C_L>0$ such that $|\Phi_{x_1, x_2}(u)-u|\leq C_L|x_1-x_2|$ for $|u|\leq L$;
   
    \hspace{15pt} ii) there exists $C>0$ such that $\|\nabla \Phi_{x_1, x_2}-I\|_{L^\infty}\leq C|x_1-x_2|$.
\end{enumerate}

Condition 1 is necessary in order that the constraint $u\in F_x$ on the perforation be closed. Condition~2 implies some kind of connectedness at infinity for $F_x$ uniformly in $x$, and is trivially satisfied if $F_x$ are contained in a common compact set. Finally Condition 3 is a regularity assumption on $x\mapsto F_x$, and is trivially satisfied if $F_x$ is independent of $x$.

\bigskip
\noindent{\bf Functional setting.} By the highly inhomogeneous nature of our problems, we will use topologies that allow for increasingly oscillating functions. Our functional setting will be then the vector Sobolev space $W^{1,p}(\Omega;\mathbb R^m)$ equipped with its weak topology. Since all our energies will be equi-coercive, we can equivalently  equip this space with the strong convergence in $L^p(\Omega;\mathbb R^m)$. Note that embedding theorems give that functions in $W^{1,p}(\Omega;\mathbb R^m)$ are continuous if $p>d$, which explains why constraints give a trivial effect in that case. For example, a homogeneous Dirichlet condition will be only satisfied in the limit by the function identically $0$. The {\em critical case} $p=d$ can also be treated but at the expense of a heavier notation and longer proofs (see the observations at the end of this section).

\bigskip
We subdivide our result in different cases depending on the geometry and location of the perforation. In the first one, we consider the case of periodically distributed perforations in the interior of $\Omega$, giving a limit bulk contribution. By modifying this case we will then obtain different statements for locally periodic perforations and for perforations located on the boundary of $\Omega$.

\smallskip
We first consider functions $u\colon\Omega\to\mathbb R^m$, with constraints given on a perforation composed of an $\varepsilon$-periodic array of closed sets.
As a first prototypical case, we can think of $d=2$, $m=3$ and $F_x$, independent of $x$, a cone, describing a membrane subjected to a unilateral constraint. We can take for example $F_x=\{z\in\mathbb R^3, z_3\ge 0\}$, in which displacements below the plane $\{z_3=0\}$ are forbidden in the perforation, or $F_x=\{w(x)\}+F$, where we penalize functions such that $u(x)-w(x)$ does not lay in the set $F$.
As a second example we may take $m=3$ and $d=2$ or $3$ with constraints $u(x)=w(x)$ on $x$ in a regular $d$-dimensional lattice of $\mathbb R^d$. The function $w$ may be thought of as a control variable, modeling a control device internal to the reference set $\Omega$.
The constraint is relaxed to a bulk integral of the form
$$
\int_\Omega \varphi(x,u(x))\,\d x;
$$
that is, for $\varepsilon$ and $\delta$ small the functions $u$ minimizing the energy with such constraints, subjected to given boundary and/or forcing conditions, will be close to minimizers of
$$
\int_\Omega \sigma(\nabla u)\d x+\int_\Omega \varphi(x,u(x))\,\d x,
$$
subjected to the same boundary and/or forcing conditions.
The function $\varphi(x,z)$ is characterized by a minimum problem obtained by blow-up at the perforation site, which we may assume to be  parameterized by $x$, with $z$ acting as a boundary datum.
For a function $u$ with value $u(x)=z$ at a given point $x$ and for a sequence $u_\varepsilon$ converging to $u$ the condition $u_\varepsilon(x)=z$ is in general in contrast with the constraint that $u_\varepsilon\in F_x$ on the element of the perforation parameterized by $x$. Nevertheless it is not restrictive to suppose that the equality $u_\varepsilon=z$ is (approximately) satisfied close to the perforation. Hence, this last condition is interpreted as a boundary condition; that is, we may assume that, given a target function $u$, the approximating functions $u_\varepsilon$ at given $\varepsilon$ satisfy $u_\varepsilon=u(x)$ on a ball of radius $\delta R$ (with $R$ large) around the corresponding element of the perforation. After scaling the perforation in order to parameterize its effect on the fixed reference set $K$, this boundary condition is given on some large ball $B_R$.
Hence, minimizing on all function with this given boundary condition, we define $\varphi_{\varepsilon, R}$ as 
\begin{equation}\label{eq: cap er}
    \varphi_{\varepsilon, R}(x, u)\coloneqq \inf\left\{ \int_{B(0,R)} \sigma_\varepsilon(\nabla \varv)\, \d z, \text{ where } \varv-u\in W^{1, p}_0(B(0, R); \R^m), \varv(z)\in F_x \hbox{ if }z\in K\right\},
\end{equation}
where $\sigma_\varepsilon(\xi)=\varepsilon^{\frac{dp}{d-p}}\sigma(\varepsilon^{-\frac{d}{d-p}}\xi)$.
The asymptotic effect of the perforation will be described by letting $\varepsilon\to 0$ and letting $R$ diverge. Recalling that under our assumptions $\sigma_\varepsilon$ tend to $\sigma_0$ as $\varepsilon$ tends to $0$, in order to describe this effect we set 
\begin{equation}\label{defpsi}
    \varphi(x, u)\coloneqq \inf\left\{ \int_{\R^d} \sigma_0(\nabla \varv)\, \d z, \text{ where } \varv-u\in W^{1, p}(\R^d; \R^m), \varv(z)\in F_x \hbox{ if }z\in K \right\}.
\end{equation}
If $m=1$, $u=1$ and $F_x=\{0\}$ then this number is a classical quantity called the {\em $p$-capacity} of the set $K$. Formula \eqref{defpsi} is a vector generalization, giving a sort of $p$-capacity of the set $K$ with respect to the constraint $F_x$. Note that $\varphi(x,u)=0$ if $u\in F_x$ and $u\mapsto\varphi(x,u)$ grows as a power $p$ of the distance from $F_x$; $\varphi(x,\cdot)$ can therefore be thought as a {\em $p$-distance from the set $F_x$.}

\smallskip
We next provide two examples concerning the characterization of the function $\varphi$. As a first example, we consider the case of $F_x=\{w(x)\}$ with $w$ a smooth function; that is, at given $\varepsilon$, we fix the value of $u$ on the perforation to be a function $w$, which we can interpret as describing the deformation of the perforation determined by a controlling device internal to the elastic body.
In this case $\varphi=\varphi_0(u(x)-w(x))$, where
\begin{equation}
    \varphi_0(u)\coloneqq \inf\left\{ \int_\Omega \sigma_0(\nabla \varv)\, \d z, \text{ where } \varv\in W^{1, p}(\R^d; \R^m), \varv(z)=u\hbox{ if }z\in K \right\};
\end{equation}
that is, recalling that $\sigma_0$ is positively homogeneous of degree $p$, the new term penalizes a kind of $p$-th power of the distance of $u$ from $w$.

In the second example, we confine the perforation to a fixed set, independent of $x$; that is, $F_x=F$ independent of $x$ and $\sigma_0(\xi)=\|\xi\|^2$. In this case $\varphi(u)=c\,{\rm dist}^2(u, F)$. In particular if $F=\mathbb S^{d-1}$ then $\varphi(u)=c(|u|-1)^2$. Note that, in this case, we obtain in the limit a Ginzburg--Landau functional as a limit of Dirichlet energies with constraints on perforations. If $d=2$ then the power $p=2$ does not satisfy the hypotheses of our theorem below, which nevertheless can be proved under an exponential scaling for the perforation.
\smallskip 

The following result is the prototype on which the following ones are modelled.  
The reason why the size of the perforation is $\delta_\varepsilon=\varepsilon^{d/d-p}$ is that, following the scaling procedure described above, in this regime the minimal energy of each perforation at a point $x$ is $\varepsilon^d\varphi(x,u(x))$. Since the perforations are distributed on a cubic lattice of size $\varepsilon$, the summation of these terms produces the integral $\int_\Omega \varphi(x,u(x))\,\d x$.
Again we note that the statement as a theorem concerning the computation of a $\Gamma$-limit is a useful shorthand to express the fact that problems involving energies $\mathcal F_\varepsilon$ in \eqref{eq: thm1-1} with given forcing terms and boundary conditions converge to the corresponding problems for $\mathcal F$ in \eqref{eq: thm1-2}. The only care, as boundary conditions are concerned is that they should not be in contrast with the constraint on the perforation. This technical point is achieved either if the boundary values are compatible with the constraints, or, for example, if elements of the perforation at scale of order $\varepsilon$ from the boundary are ignored, a condition which does not change the claim of the theorem.

\begin{thm}[Limit analysis for periodic bulk perforations]\label{thm1}
    Let $\sigma$ and $x\mapsto F_x$ satisfy the assumptions stated in the previous paragraphs, let $1<p<d$, and let $\delta_\varepsilon=\varepsilon^{{d}/{d-p}}$. If we set 
    $x^\varepsilon_i=\varepsilon i$ for all $i\in\mathbb Z^d$ such that $x_i^\varepsilon+ \delta_\varepsilon K\subset\Omega$, and let
    \begin{equation}\label{eq: thm1-1}
      \F_\varepsilon(u)\coloneqq\begin{cases}
            \displaystyle\int_\Omega \sigma(\nabla u)\d x \quad &\text{if } u(x)\in F_{x_i^\varepsilon}\text{ for all }i\hbox{ and } x\in {x_i^\varepsilon+ \delta_\varepsilon K},\\
            +\infty & \text{otherwise},
        \end{cases} 
    \end{equation}
    for $u\in W^{1,p}(\Omega;\mathbb R^m)$; then the $\Gamma$-limit of $\F_\varepsilon$ as $\varepsilon\to 0$ in the $L^p$ and weak $W^{1,p}$ topologies is
      \begin{equation}\label{eq: thm1-2}
   \F(u)\coloneqq \int_{\Omega}\sigma(\nabla u)\d x+\int_{\Omega}\varphi(x, u(x))\,\d x,
    \end{equation}
    with $\varphi$ in \eqref{defpsi},
    with domain $W^{1,p}(\Omega;\mathbb R^m)$.
\end{thm}

This result and the following are stated in terms of the displacement $u$, but can equivalently stated for energies and constraints $y(x)\in F_x$ formulated in terms of the deformation $y$. More precisely, keeping the same notation for the constraint, we can consider energies defined by 
   \begin{equation}
{\mathcal G}_\varepsilon(y)\coloneqq\begin{cases}
      \displaystyle\int_\Omega \sigma(\nabla y)\d x \quad &\text{if } y(x)\in F_{x_i^\varepsilon}\text{ for all }i\hbox{ and } x\in {x_i^\varepsilon+ \delta_\varepsilon K},\\
            +\infty & \text{otherwise},
        \end{cases} 
    \end{equation}
    for $y\in W^{1,p}(\Omega;\mathbb R^m)$, and prove that
the $\Gamma$-limit of ${\mathcal G}_\varepsilon$ as $\varepsilon\to 0$ is
      \begin{equation}
   {\mathcal G}(y)\coloneqq \int_{\Omega}\sigma(\nabla y)\d x+\int_{\Omega}\varphi(x, y(x))\,\d x
    \end{equation}
    with domain $W^{1,p}(\Omega;\mathbb R^m)$.
    
Indeed, replacing $u(x)=y(x)-x$ in $\mathcal G_\varepsilon$, we can define
   \begin{equation}    \widetilde\F_\varepsilon(u)\coloneqq\begin{cases}          \displaystyle\int_\Omega \sigma(\nabla u-id)\d x \quad &\text{if } u(x)\in F_{x_i^\varepsilon}-x\text{ for all }i\hbox{ and } x\in {x_i^\varepsilon+ \delta_\varepsilon K},\\
            +\infty & \text{otherwise},
        \end{cases} 
    \end{equation}
    whose $\Gamma$-limit is proven, with a small-perturbation argument, to be  the same as that of
   \begin{equation}    \F_\varepsilon(u)\coloneqq\begin{cases}          \displaystyle\int_\Omega \sigma(\nabla u-id)\d x \quad &\text{if } u(x)\in F_{x_i^\varepsilon}-x_i^\varepsilon \text{ for all }i\hbox{ and } x\in {x_i^\varepsilon+ \delta_\varepsilon K},\\
            +\infty & \text{otherwise},
        \end{cases} 
    \end{equation}
    which is in the form required to apply Theorem \ref{thm1}.
Noting that the recession function of $\xi\mapsto\sigma(\xi-id)$ is the same as that of $\sigma$, defining $\varphi$ as in \eqref{defpsi}, Theorem \ref{thm1} gives the $\Gamma$-limit corresponding to the constraints $F_x-x$; that is,
     \begin{equation}
   \F(u)= \int_{\Omega}\sigma(\nabla u -id)\d x+\int_{\Omega}\varphi(x, u(x)-x)\,\d x=\mathcal G(y),
    \end{equation}
and the claim.

\smallskip

We give a proof of Theorem \ref{thm1}. As usual, the computation of the $\Gamma$-limit is split into a lower bound, giving an ansatz-free estimate on the energy of an arbitrary sequence converging to $u$, and an upper bound, which amounts to the construction of a {\em recovery sequence} along which the lower bound is matched (up to an arbitrarily small error) (see \cite{GCB}). The construction of the recovery sequence is linked in our case to the definition of the term $\varphi$, and is obtained by taking, loosely speaking, scaled optimal functions for the problem defining $\varphi(x,u(x))$ close to the perforation, and the function $u(x)$ far from the perforation.

\begin{proof}[Proof of Theorem {\rm \ref{thm1}}]
We only give the main steps of the proof, confining the technicalities to some results in dedicated lemmas reported in Appendix~\ref{app:lemmas}.

    {\em Lower bound}.
    Let $u_\varepsilon$ be an arbitrary sequence converging to $u$ in $W^{1, p}(\Omega,\mathbb R^m)$ and let $u_{\varepsilon_h}$ be a subsequence realizing $\liminf_{\varepsilon\to 0} \mathcal F_\varepsilon(u_\varepsilon)$. With fixed $\eta>0$, by Lemma \ref{lemma: truncation} there exist $\varv_h$  converging weakly to some $\varv$ in $W^{1, p}(\Omega,\mathbb R^m)$, $|\varv_h|\leq R_\eta$, 
    \begin{equation}\label{eq: stima p}
        \|\varv-u\|_{W^{1, p}}=o_{\eta}(1),
    \end{equation}
    and \begin{equation} \label{eq: stima energia}
        \liminf_{h\to+\infty} \mathcal F_{\varepsilon_h}(\varv_h)\leq \liminf_{\varepsilon\to 0} \mathcal F_\varepsilon(u_\varepsilon)+o_\eta(1)
    \end{equation}
    as $\eta\to 0$.
    Moreover, by Lemma \ref{lemma: change bd}, with fixed $N\ge 1$ we can assume that $\varv_h$ is constant on spherical surfaces $S^{\varepsilon_h}_j$ containing the perforation (see the statement of the lemma for a precise definition, choosing $\rho_\varepsilon$ as in \ref{lemma: pw}) indexed by $j\in\{1,\ldots, N\}$, up to replacing \eqref{eq: stima p} with 
        $\|\varv-u\|_{W^{1, p}}=o_{\eta, N}(1)$ 
    and \eqref{eq: stima energia} with 
    $\liminf\limits_{h\to+\infty} \mathcal F_{\varepsilon_h}(\varv_h)\leq \liminf\limits_{\varepsilon\to 0} \mathcal F_\varepsilon(u_\varepsilon)+o_{\eta, N}(1)$
    as $\eta\to 0$ and $N\to+\infty$.

    Let now $w_h$ be defined by extending $\varv_h$ as the corresponding constant in the balls whose boundaries are the spherical surfaces $S^{\varepsilon_h}_j$.
    Note that $w_h$ still converge to $\varv$.
    Setting $R_\varepsilon=\frac{\rho_{\varepsilon}}{\delta_\varepsilon}$ and $\varphi_h(x, z)=\varphi_{\varepsilon_h, R_{\varepsilon_h}}(x, z)$ as in \eqref{eq: approx cap} we can now estimate 
    \begin{align}
        \liminf_{h\to+\infty} \mathcal F_{\varepsilon_h}(\varv_h)&\geq \liminf_{h\to+\infty} \int_{\Omega\setminus \bigcup B^{\varepsilon_h}_j}\sigma(\nabla \varv_h)\,\d x+\sum_j\int_{B^{\varepsilon_h}_j} \sigma(\nabla \varv_h)\,\d x\\
        &\geq \liminf_{h\to+\infty}\int_\Omega \sigma(\nabla w_h)\,\d x+\liminf_{h\to+\infty} \sum_j \varepsilon_h^d \varphi_{\varepsilon_h, 2^{-2 i_{\varepsilon_h}(j)-1}\frac{\rho_{\varepsilon_h}}{\delta}}(\varepsilon_h j, \bar{\varv}^h_j)\\
        &\geq\int_\Omega \sigma(\nabla \varv)\,\d x+\liminf_{h\to+\infty} \int_\Omega \varphi_{\varepsilon_h, \frac{\rho_{\varepsilon_h}}{\delta}}\bigg(\sum_j \mathbbm{1}_{\varepsilon_h \left(j+(-\frac{1}{2}, \frac{1}{2}\right)^d)} \varepsilon_h j, \sum_j \mathbbm{1}_{\varepsilon_h \left(j+(-\frac{1}{2}, \frac{1}{2})^d\right)}\bar{\varv}^{\varepsilon_h}_j\bigg)\,\d x
        \\
        &\geq \int_\Omega \sigma(\nabla \varv)\,\d x +\liminf_{h\to+\infty}\int_\Omega \varphi_h\bigg(\sum_j \mathbbm{1}_{\varepsilon_h \left(j+(-\frac{1}{2}, \frac{1}{2}\right)^d)} \varepsilon_h j, \sum_j \mathbbm{1}_{\varepsilon_h \left(j+(-\frac{1}{2}, \frac{1}{2})^d\right)}\bar{\varv}^{\varepsilon_h}_j\bigg)\,\d x. \label{eq: lower bound 1}
    \end{align}
    Now, by Lemma \ref{lemma: change bd} we know that
       $ \sum_j \mathbbm{1}_{\varepsilon_h \left(j+(-\frac{1}{2}, \frac{1}{2})^d\right)}\bar{\varv}^{\varepsilon_h}_j$ converges in $L^p(\Omega;\mathbb R^m)$ to $\varv$,
    and by Lemma \ref{lemma: uniform} $\varphi_h$ converges to $\varphi$ uniformly on $\overline{\Omega}\times B(0, L)$. In particular, this gives 
    \begin{equation}\label{eq: lower bound 2}
        \lim_{h\to+\infty} \int_\Omega \varphi_h\bigg(\sum_j \mathbbm{1}_{\varepsilon_h \big(j+(-\frac{1}{2}, \frac{1}{2})^d\big)} \varepsilon_h j, \sum_j \mathbbm{1}_{\varepsilon_h \left(j+(-\frac{1}{2}, \frac{1}{2})^d\right)}\bar{\varv}^{\varepsilon_h}_j\bigg)\,\d x= \int_\Omega \varphi(x, \varv(x))\,\d x.
    \end{equation}
    Following on from \eqref{eq: lower bound 1} and applying \eqref{eq: lower bound 2} we conclude
    \begin{equation}
        \liminf_{\varepsilon\to 0} \mathcal F_\varepsilon(u_\varepsilon)\geq \liminf_{h\to+\infty} \mathcal F_{\varepsilon_h}(\varv_h)-o_{\eta, N}(1)\geq \mathcal F(v)- o_{\eta, N}(1).
    \end{equation}
    By letting $\eta$ to 0 and $N$ to $+\infty$, we finally deduce that
$\mathcal F(u)\leq \liminf\limits_{\varepsilon\to 0} \mathcal F_\varepsilon(u_\varepsilon).$

{\em Upper bound}.
Let $u\in W^{1, p}(\Omega;\R^m)$ to be approximated. Without loss of generality, we may assume that $u$ is bounded, say $|u|\leq L$.
With fixed $N>1$, we can apply Lemma \ref{lemma: change bd} applied to $u_\varepsilon= u$ for all $\varepsilon$ and $\rho_\varepsilon$ as in Lemma \ref{lemma: pw},
obtaining a sequence $(\varv_\varepsilon)_{\varepsilon>0}$ constant on the boundary of some balls $B^\varepsilon_{i, j}$ close to the perforation.

We now modify the sequence $\varv_\varepsilon$ in the balls $B^\varepsilon_{i, j}$ in order to produce the required recovery sequence. 
Let $\eta>0$; by Lemma \ref{lemma: uniform} there exists $\varepsilon_0$ small enough such that $|\varphi_{\varepsilon, R_\varepsilon^i}(x, u_0)-\varphi(x, u_0)|\leq \eta$ for $|u_0|\leq L$  and $0<\varepsilon<\varepsilon_0$,
with $R_\varepsilon^i=2^{-2 i_\varepsilon(j)-1}\frac{\rho_\varepsilon}{\delta_\varepsilon}$ for $i=1, \ldots, N$.
Let now $\varepsilon<\varepsilon_0$ and $j\in\mathbb Z^d$ be such that $B^\varepsilon_{i, j}\in \Omega$. Let $\zeta_\varepsilon^{i, j}$ be a valid competitor for \eqref{eq: cap er} such that 
\begin{equation}
    \int_{B(0, R^i_{\varepsilon})} \sigma_\varepsilon(\nabla \zeta_\varepsilon^{i, j})\,\d z\leq \varphi_{\varepsilon, R_\varepsilon^i}(x, \bar{\varv}_j^\varepsilon)+\eta.
\end{equation}
Finally, for $\varepsilon<\varepsilon_0$, define 
\begin{equation}
    u_\varepsilon(x)=
    \begin{cases}
        \zeta_\varepsilon^{i, j}(\delta_\varepsilon (x-\varepsilon j) )\quad&\text{ if }x\in B^\varepsilon_{i, j},\\
        \varv_\varepsilon(x)\quad&\text{ otherwhise.}
    \end{cases}
\end{equation}

We conclude by proving that $u_\varepsilon$ is a recovery sequence. It holds
\begin{align}
    \limsup_{\varepsilon \to 0}\int_\Omega \sigma(\nabla u_\varepsilon)&\leq \limsup_{\varepsilon \to 0}\int_{\Omega\setminus \bigcup_j B^\varepsilon_{i, j}}\sigma(\nabla \varv_\varepsilon)\,\d z+\limsup_{\varepsilon\to 0}\sum_j \int_{B^\varepsilon_{i, j}}\sigma(\nabla u_\varepsilon)\,\d z\\
    &=\int_\Omega \sigma(\nabla u)\,\d z+ \limsup_{\varepsilon\to 0}\varepsilon^d \sum_j \int_{B(0, R^i_\varepsilon)}\sigma_\varepsilon(\nabla \zeta_\varepsilon^{i, j})\,\d z
\end{align}
where the last equality comes from the equi-integrability of $|\nabla \varv_\varepsilon|^p$.
Now
\begin{align}
    \limsup_{\varepsilon\to 0}\varepsilon^d \sum_j \int_{B(0, R^i_\varepsilon)}\sigma_\varepsilon(\nabla\zeta_\varepsilon^{i, j} )\,\d z&\leq \limsup_{\varepsilon\to 0}\bigg(\varepsilon^d \sum_j \varphi_{\varepsilon, R_\varepsilon^i}(\varepsilon j, \bar{\varv}_j^\varepsilon)+\eta|\Omega+B(0, \varepsilon)|\bigg)\\
    &\leq \limsup_{\varepsilon\to 0} \int_\Omega \varphi\bigg(\sum_j \mathbbm{1}_{\varepsilon \left(j+(-\frac{1}{2}, \frac{1}{2}\right)^d)} \varepsilon j, \sum_j \mathbbm{1}_{\varepsilon \left(j+(-\frac{1}{2}, \frac{1}{2})^d\right)}\bar{\varv}^{\varepsilon}_j\bigg)\,\d z\\
    &\hskip2cm+ 2\eta\,|\Omega+B(0, \varepsilon)|.
\end{align}
By recalling \eqref{eq: lower bound 2} we conclude that
    $\limsup\limits_{\varepsilon\to 0} \mathcal F_\varepsilon(u_\varepsilon)\leq \mathcal F(u)+C\eta.$
By sending $\eta\to 0$ we conclude.
\end{proof}

An interesting variation on the previous result is obtained when perforations are not distributed uniformly but have a local limit density described by a (bounded) function $\rho$. In this case, the relaxed constraint takes the form
$$
\int_\Omega \varphi(x,u(x))\,\rho(x)\,\d x,
$$
We state a version of this result obtained by deforming the reference lattice on which the perforation is parameterized. This deformation is obtained by introducing a possible local variation of the oscillations and orientation described by a smooth diffeomorphism. In this case the density $\rho$ is given in terms of the Jacobian determinant of the transformation.
We note that we could treat distributions of the perforations with much less regularity, and also random distributions, upon requiring that perforations cannot accumulate or be too sparse. For simplicity of notation, we limit to considering perforations that can be parameterized on cubic lattices, since a precise statement in the general case would need some technicalities.

\begin{thm}[Limit analysis for locally periodic bulk perforations]\label{thm2}
 Let $\Psi\colon\R^d\to\R^d$ be a smooth diffeomorphism and set $$K^\varepsilon_i=\Psi(x^\varepsilon_i)+\delta_\varepsilon K$$ 
 for all $i\in\mathbb Z^{d}$ such that $K^\varepsilon_i\subset\Omega$, (that is, we center the elements of the perforation in the set obtained as the image of a cubic lattice by $\Psi$).
 Let the assumptions on $f$ and $x\mapsto F_x$ be as in Theorem \ref{thm1}, let $1<p<d$, let $\delta_\varepsilon=\varepsilon^{\frac{d}{d-p}}$, and let 
 \begin{equation}
\F_\varepsilon(u)\coloneqq\begin{cases}
            \displaystyle\int_\Omega \sigma(\nabla u)\d x \quad &\text{if } u(z)\in F_{\Psi(x_i^\varepsilon)}\text{ for all }i \hbox{ and } z\in K^\varepsilon_i,\\
            +\infty & \text{otherwise}.
        \end{cases} 
    \end{equation}
    Then the $\Gamma$-limit of $\F_\varepsilon$ as $\varepsilon\to 0$ in the $L^p$ and weak $W^{1,p}$ topologies is
    \begin{equation}
   \F(u)\coloneqq \int_{\Omega}\sigma(\nabla u)\d x+\int_{\Omega}\varphi(x, u(x))\,\rho(x)\,\d x,
    \end{equation}
       with domain $W^{1,p}(\Omega;\mathbb R^m)$,
    where $\rho(x)=|{\rm det}(\nabla\Psi(\Psi^{-1}(x)))|^{-1}$.
\end{thm}

\begin{proof}[Proof of Theorem {\rm \ref{thm2}}]
The proof of Theorem \ref{thm2} is obtained following the one of Theorem \ref{thm1} by noting that after a change of variables the term $|{\rm det}(\nabla\Psi(\Psi^{-1}(x)))|^{-1}$ is (approximately) a prefactor of the contribution of the single
perforation at $x$. The construction of a recovery sequence is the same, upon noting that $\rho(x)$ so defined is the asymptotic density of perforations at $x$.
\end{proof}

Note that the hypotheses on the lattice can be localized by assuming that there exist disjoint open subsets $\Omega_j$ of $\Omega$ such that $|\Omega\setminus \bigcup_j\Omega_j|=0$, and $\Psi_j\colon\R^d\to\R^d$ smooth diffeomorphisms such that $\mathscr L_\varepsilon\cap \Omega_j=\Psi_j(\varepsilon\mathbb Z^d)\cap \Omega_j$.
The statement of the claim is the same with $\Psi=\Psi_j$ on $\Omega_j$.

\bigskip
We now take into account 
cases when the constraint can also be imposed on a lower-dimensional (possibly deformed) set. This situation models for example control devices on low-dimen\-sio\-nal manifolds. We only state the case of a $(d-1)$-dimensional set, parameterized by a smooth hypersurface $\Sigma\subset\Omega$ along which the sites of the perforation are approximately a $(d-1)$-dimensional cubic lattice of spacing $\varepsilon$. In this case, the size of the perforation is $\delta_\varepsilon=\varepsilon^{(d-1)/(d-p)}$. With this change in the exponent, the energy of a perforation at a point $x$ for an optimal approximation is $\varepsilon^{d-1}\varphi(x,u(x))$.  The sum of these terms produces a $(d-1)$-dimensional integral on $\Sigma$, and the relaxed constraint takes the form
$$
\int_\Sigma \varphi(x,u(x))\,\rho(x)\,\d \mathcal H^{d-1},
$$
where $\mathcal H^{d-1}$ is the $(d-1)$-dimensional Hausdorff measure (for a smooth surface, coinciding with the usual surface measure). Again, $\rho(x)$ can be interpreted as a $(d-1)$-dimensional density of perforations at $x$.

\begin{thm}[Limit analysis for perforations along an internal hypersurface]\label{thm3}
    Let $\delta_\varepsilon=\varepsilon^{\frac{d-1}{d-p}}$, let $\omega$ be an open bounded set in $\mathbb R^{d-1}$ and let $\Psi\colon\omega\to\Omega$ be a smooth diffeomorphism between $\omega$ and $\Psi(\omega)$. Set \begin{equation}
    K^\varepsilon_i=\Psi(\varepsilon i)+\delta_\varepsilon K,
        \end{equation}
for all $i\in\mathbb Z^{d-1}$ such that $K^\varepsilon_i\subset\Omega$, and let 
    \begin{equation}     \F_\varepsilon(u)\coloneqq\begin{cases}
            \displaystyle\int_\Omega \sigma(\nabla u)\d x \quad &\text{if } u(z)\in F_{\Psi(x_i^\varepsilon)}\text{ for all }i \hbox{ and } z\in{K^\varepsilon_i},\\
            +\infty & \text{otherwise}.
        \end{cases} 
    \end{equation}
    Then the $\Gamma$-limit of $\F_\varepsilon$ as $\varepsilon\to 0$ in the $L^p$ and weak $W^{1,p}$ topologies is
    \begin{equation}
    \F(u)\coloneqq \int_{\Omega} \sigma(\nabla u)\,\d x+\int_{\Psi(\omega)}\varphi(x, u(x))\rho(x)\,\d\mathcal{H}^{d-1},
    \end{equation}
with domain $W^{1,p}(\Omega;\mathbb R^m)$, where $\rho(x)=|{\rm det}(\nabla\Psi(\Psi^{-1}(x)))|^{-1}$.
The function $u$ on $\Psi(\omega)$ is meant in the sense of traces.
\end{thm}

\begin{proof}[Proof of Theorem {\rm \ref{thm3}}]
The proof follows the ones of the previous two theorems, the only change being the observation that, as noted above, the scaling produces that each element of the perforation can be considered as a part of an approximation of a surface integral. The construction of the recovery sequence is the same as before.
\end{proof}

Note that a particular case is when we take as $\omega$ a parameterization of a planar section of $\Omega$, which, up to a change in coordinates, we can suppose to be $\{x_d=0\}\cap\Omega$. After identifying $\{x_d=0\}$ with $\mathbb R^{d-1}$ and taking $\Psi$ the identity, the perforation is simply $K^\varepsilon_i=\varepsilon i+\delta_\varepsilon K$,  and the $\Gamma$-limit can be written as 
$$
\mathcal F(u)\coloneqq \int_{\Omega} \sigma(\nabla u)\,\d x+\int_{\{x_d=0\} \cap\Omega}\varphi(x, u(x))\,\d x,
$$
where the last integral is performed with respect to the Lebesgue measure in $\mathbb R^{d-1}$.

\bigskip
Finally, an interesting variant is when the constraint may be imposed on the boundary of $\Omega$ (or part of it), for which the relaxed term is 
$$
\int_{\partial\Omega} \overline\varphi(x,u(x))\,\d \mathcal H^{d-1}.
$$
The difference with the perforation in the bulk, is that perforations on the boundary only take into account the deformation in the interior, and hence, assuming the boundary smooth, the scaling analysis leading to $\varphi$ must be modified, obtaining a different energy density.

We can think of such constraints as loose confinement conditions imposing for example that the points on the boundary may only lay in the interior of a given set. We note that the computations for the function $\overline\varphi$ differ from those of the aforementioned $\varphi$ in that conditions on the boundary of $\Omega$ approximately only take into account the half-space delimited by the tangent plane. 

In the following results, values of functions on points in $\partial\Omega$ are understood in the sense of traces. We consider two cases: (a) bulk perforations centered in points on the boundary and (b) perforations as subsets of the boundary. The form of $\overline\varphi$ is different in the two cases (see \eqref{eq:fibar1} and \eqref{eq:fibar2}, respectively).

\begin{thm}[Limit analysis for perforations along the  boundary]\label{thm4}
Let $A_j$ be relatively disjoint open subsets of $\partial\Omega$ and $\omega_j$ open subsets of $\R^{d-1}$ such that there exist smooth onto diffeomorphisms $\Psi_j\colon\omega_j\to A_j$. Let $\delta_\varepsilon=\varepsilon^{(d-1)/(d-p)}$ and set  either

{\rm(a)}
$K^\varepsilon_{j,i}=\Psi_j(\varepsilon i)+\delta_\varepsilon B_1\cap\Omega$ (bulk perforations centered on the boundary), or

{\rm(b)}
$K^\varepsilon_{j,i}=\Psi_j(\varepsilon i+\delta_\varepsilon B_1)$ (perforations on the boundary).

\noindent Let
\begin{equation}     \F_\varepsilon(u)\coloneqq\begin{cases}
            \displaystyle\int_\Omega \sigma(\nabla u)\,\d x \quad &\text{if } u(z)\in F_{\Psi_j(x_i^\varepsilon)}\text{ for all }i, j \hbox{ and } z\in{K^\varepsilon_{j,i}},\\
            +\infty & \text{otherwise}.
        \end{cases} 
    \end{equation}
    Then the $\Gamma$-limit of $\F_\varepsilon$ as $\varepsilon\to 0$ in the $L^p$ and weak $W^{1,p}$ topologies is
    with $\F(u)$ given by
   \begin{equation}
       \F(u)\coloneqq \int_{\Omega} \sigma(\nabla u)\,\d x+\sum_j\int_{A_j}\overline\varphi(x, u(x))\,\rho(x)\,\d\mathcal{H}^{d-1},
    \end{equation}    
      with domain $W^{1,p}(\Omega;\mathbb R^m)$, 
    where $\rho(x)=|{\rm det}(\nabla\Psi_j(\Psi_j^{-1}(x)))|^{-1}$. In case {\rm(a)}, $\overline\varphi$ is defined by
\begin{equation}\label{eq:fibar1}
\overline\varphi(x,u)=\inf\bigg\{\int_{\mathbb R^d_+}\sigma_0(S_x\nabla \varv(z))\d z:  \varv-u\in W^{1,p}(\mathbb R^d_+;\mathbb R^m),  \varv\in F_x \hbox{ on }B_1\cap\{x_d>0\}\bigg\},
    \end{equation}
    where $R^d_+=\{x\in\mathbb R^d: x_d>0\}$ and $S_x$ is a rotation carrying the outer normal $\nu(x)$ to $\Omega$ in $x$ to the vector $(0,\ldots, 0,-1)$,
    while $\overline\varphi$ is defined by
\begin{equation}\label{eq:fibar2}
       \overline\varphi(x, u)=\inf\bigg\{\int_{\mathbb R^d_+}\sigma_0(S_x\nabla  \varv(z))\d z:  \varv-u\in W^{1,p}(\mathbb R^d_+;\mathbb R^m), \varv\in F_x \hbox{ on }B_1\cap\{x_d=0\} \bigg\}
           \end{equation}
            in case  {\rm(b)}.
\end{thm}

Note that if $\sigma_0$ is a convex, reflection-invariant function then $\overline\varphi(x, z)$ satisfies 
\begin{equation}
    \overline\varphi(x,u)=\frac{1}{2}\varphi(x, u)=\frac{1}{2}\inf\bigg\{\int_{\mathbb R^d}\sigma_0(\nabla \varv(z))\,\d z: \varv-u\in W^{1,p}(\mathbb R^d;\mathbb R^m), \varv\in F_x \hbox{ on }B_1\bigg\}
\end{equation}
in case (a) and analogously in case (b).

\begin{proof}[Proof of Theorem {\rm \ref{thm4}}] The proof differs from the one of the previous theorem in that the blow-up argument close to the boundary produces limit minimum problem defined on half-spaces.
\end{proof}

\medskip

As examples of relevant constraints on the boundary we can consider conditions related to the normal direction. For instance, $F_x=\{t\nu_x:t\geq 0\}$ where $\nu_x$ is the outward normal to $\partial \Omega$. This is a sort of infinitesimal constraint, allowing energy-free displacement only in the normal direction to $\Omega$. Another example is that in which $F_x$ is the cone $\{z\in\mathbb R^{d}:\langle\nu_x, z\rangle \geq \kappa|z|\}$. In particular, for $\kappa=0$ this allows for
energy-free displacements if they are not towards $\Omega$.
In both cases, $F_x$ are cones and, if $\sigma(\xi)=\|\xi\|^p$, then the term ${\rm dist}^p(z, F_x)$ in $\overline\varphi$ can be interpreted as a power of the norm of the projection on the dual cone.

\bigskip\noindent
{\bf The critical case $p=d$.} The case $p=d$ cannot directly be treated as when $p<d$.
The reason is that capacitary problems for positively $d$-homogeneous functions are invariant by scaling. As a consequence, the scaling parameter $\delta=\varepsilon^{d/d-p}$ is meaningless and cannot be used. At this scale the perforation size does not scale as a power of the period, but is exponentially small, and we have to take $\delta= e^{-\kappa/\varepsilon^{d/(d-1)}}$. In accord with the case when $F_x=\{0\}$ (homogeneous Dirichlet conditions), which has been treated in \cite{sig}, in Theorem \ref{thm1} the corresponding $\varphi=\kappa^{1-d}\varphi_1$ is described by the asymptotic formula
$$
\varphi_1(x,u)= \lim_{T\to+\infty}|\log T|^{d-1}
\min\bigg\{\int_{B(0,T)} \sigma_0(\nabla \varv(z))\,\d z: \varv-u\in W^{1,p}_0(B(0,T);\mathbb R^m),\ \varv\in F_x\hbox{ on }B(0,1)
\bigg\}.
$$
Note that the form of $\varphi$ does not depend on $K$, but it depends on the exponential decay through the factor $\kappa^{1-d}$.
As in the case of Dirichlet conditions we have a
discontinuous dependence of $\varphi$ on $p$ at the critical scaling, which suggests a fine  dependence of the perforation on $p$ close to $d$ at small but finite $\varepsilon$ in the spirit of {\em expansions by $\Gamma$-convergence} \cite{Br-Sig,BT}.

\section{Numerical examples}\label{sec:num_res}

In this section, we illustrate the application of the mathematical results detailed above through numerical simulations. We focus on problems such that the reference configuration $\Omega$ is a subset of $\R^2$ that is mapped by the deformation to either $\R^2$ or $\R^3$. This context is of particular relevance for elastic membranes subjected to perforations at the boundary or in the interior. Specifically, we take the square domain $ (0,\ell)\times(0,\ell)$ of side length $\ell$ as the reference set  and we assume that the hyperelastic energy density is of the form $\sigma(\xi) = \|\xi\|^p$, with $p \in (1,2)$.
\begin{figure}[t]
    \centering
\includegraphics[width=1.0\linewidth]{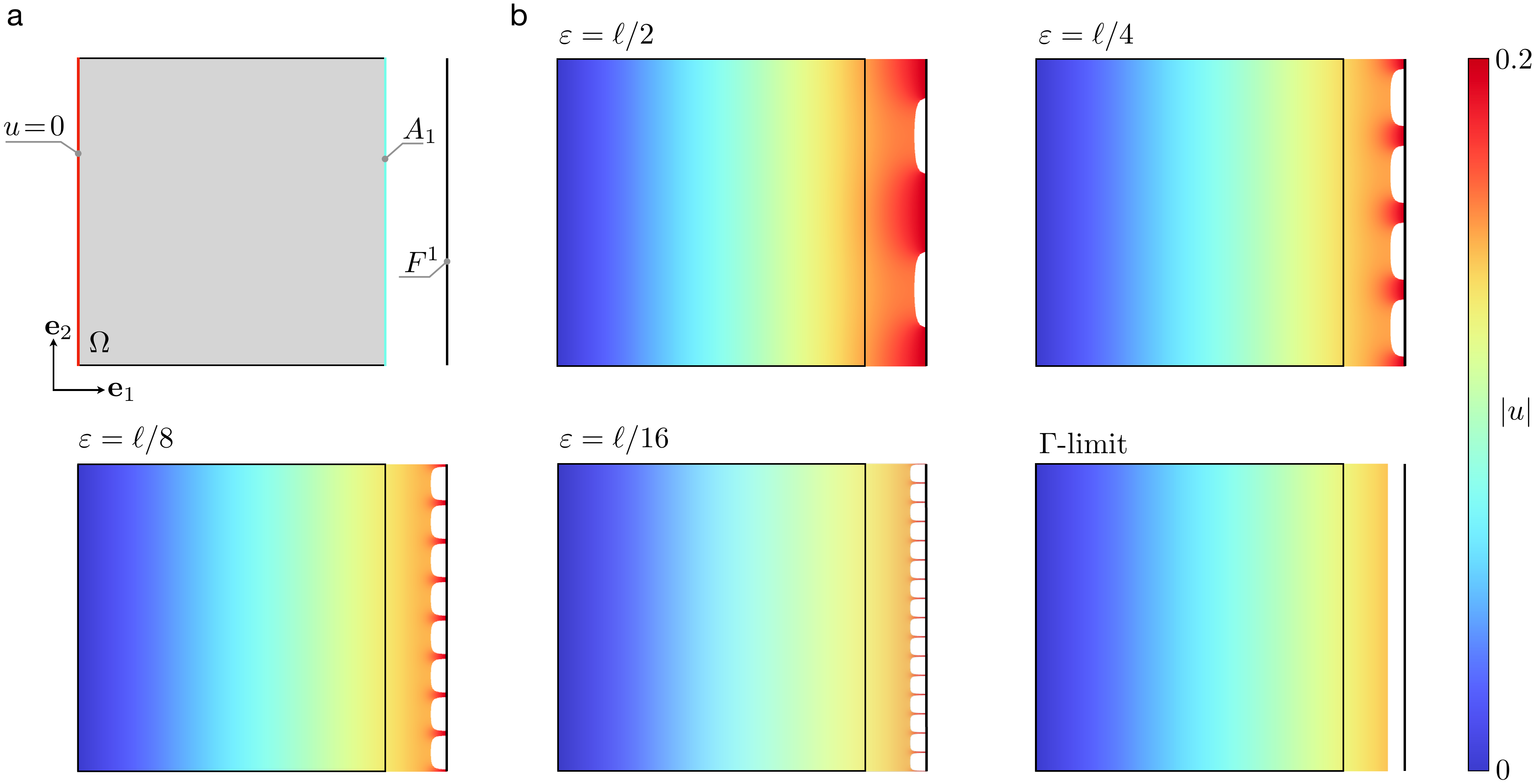}
    \caption{Numerical results for the case of a square domain $\Omega=(0, \ell)\times (0, \ell)$ of side length $\ell$ subjected to Dirichlet boundary conditions on $\{0\}\times (0, \ell)$ and to the constraint $F^1_{(x_1, x_2)}=\{(z_1, z_2): z_1 = x_{1,\textrm{c}} - x_1\}$ at the perforation $A_1=\{\ell\}\times (0, \ell)$. (a) Sketch of the boundary-value problem. (b) Deformed configurations for perforations with $\varepsilon = \{\ell/2, \ell/4, \ell/8, \ell/16\}$ and for the limit functional. The colorbar reports the norm of the displacement field. For all the results, $p = 3/2$, $\ell = 1$, and $x_{1,\textrm{c}} = 6/5$.}
    \label{fig:ex_num_1}
\end{figure}
We next consider distinct boundary-value problems as follows:

\begin{enumerate}
    \item The reference set $\Omega$ is subjected to null displacement on $\{0\}\times(0, \ell)$ and to the boundary constraint $F^1_{(x_1, x_2)}=\{(z_1, z_2): z_1=x_{1,\textrm{c}}-x_1\}$ on the perforation $A_1=\{\ell\}\times (0, \ell)$, see Fig.~\ref{fig:ex_num_1}a. Specifically, the perforation is constrained on the vertical line passing through $x_1 = x_{1, \textrm{c}}$ and
    $$\overline{\varphi}((x_1, x_2), (z_1, z_2))=c \,|z_1 + x_1 - x_{1,\textrm{c}}|^p ,$$
    according to Theorem~\ref{thm4}.
    
    \item The reference set $\Omega$ is subjected to the boundary constraint $F^2_{(x_1, x_2)}=\{(z_1, z_2): (z_1 + x_1 - x_{1,\textrm{c}})^2+(z_2 + x_2 - x_{2,\textrm{c}})^2=R^2\}$ on the perforations $A_1=\{\ell\}\times(0, \ell)$, $A_2=(0, \ell)\times \{\ell\}$, $A_3=\{0\}\times (0, \ell)$, and $A_4=(0, \ell)\times \{0\}$, see Fig.~\ref{fig:ex_num_2}a. Hence, the perforations are constrained on the circle of center $(x_{1,\textrm{c}}, x_{2,\textrm{c}})$ and of radius $R$, so that $$\overline{\varphi}((x_1, x_2), (z_1, z_2))=c \,\Big|\sqrt{(z_1 + x_1 -x_{1,\textrm{c}})^2+(z_2 + x_2 - x_{2,\textrm{c}})^2}-R\Big|^p ,$$ according to Theorem~\ref{thm4}.

    \item The reference set $\Omega$ is subjected to null displacement on $(0,\ell)\times\{0\}$ and on $(0,\ell)\times\{\ell\}$, and to the constraint $F^3_{(x_1, x_2)}=\{(z_1, z_2, z_3): (z_2 + x_2 - x_{2,\textrm{a}})^2 + (z_3)^2 = R^2\}$ on the one-dimensional perforations $I_1 = \{\ell/10\}\times (\ell/10,9\ell/10)$, $I_2 = \{\ell/2\}\times (\ell/10,9\ell/10)$, and $I_3 = \{9\ell/10\}\times (\ell/10,9\ell/10)$ in the interior of $\Omega$, see Fig.~\ref{fig:ex_num_3}a. Thus, the perforations are constrained on the cylinder of radius $R$ and axis the line $x_2 = x_{2,\textrm{a}}$, so that $$\varphi((x_1, x_2), (z_1, z_2, z_3))= c \,\Big|\sqrt{(z_2 + x_2 - x_{2,\textrm{a}})^2 + (z_3)^2}-R\Big|^p ,$$ according to Theorem~\ref{thm3}.

    \item The reference set $\Omega$ is subjected to null displacement on $\{0\}\times(0, \ell)$ and to the boundary constraint $F^4_{(x_1, x_2)}=\{(z_1, z_2): z_2 + x_2 + z_1 + x_1 - 3\ell/2 \ge 0\}$ on the perforation $A_1=\{\ell\}\times (0, \ell)$, see Fig.~\ref{fig:ex_num_4}a. Specifically, the perforation is unilaterally constrained to the right of the line $x_2 + x_1 - 3\ell/2 = 0$, so that denoting by $H$ the unit step function $$\overline{\varphi}((x_1, x_2), (z_1, z_2))=c
    \,H\big(\tfrac32\ell - z_2 - x_2 - z_1 - x_1\big) \,\big|z_2 + x_2 + z_1 + x_1 - \tfrac32\ell\big|^p 2^{-\frac{p}2} ,$$ according to Theorem~\ref{thm4}.
    
\end{enumerate}

We remark that, for the boundary-value problems outlined above, the densities $\varphi$ and $\overline{\varphi}$ appearing in the limit functional measure the distance of the perforations from the target sets.

In the following, we assume the scaling $\delta_\varepsilon=\varepsilon^{\frac{1}{2-p}}$ and compare the numerical results obtained for perforations of finite size with those corresponding to the $\Gamma$-limit.
For the case of perforations characterized by finite size, the weak formulation of the governing equations is achieved by computing the functional derivative of the Lagrangian consisting in the hyperelastic strain energy augmented by the constraints  imposed on the perforations. Likewise, solutions to the respective limit cases are achieved through the weak formulation of the functionals presented in Section~\ref{sec:mat_results}.

\begin{figure}[t]
    \centering
\includegraphics[width=1.0\linewidth]{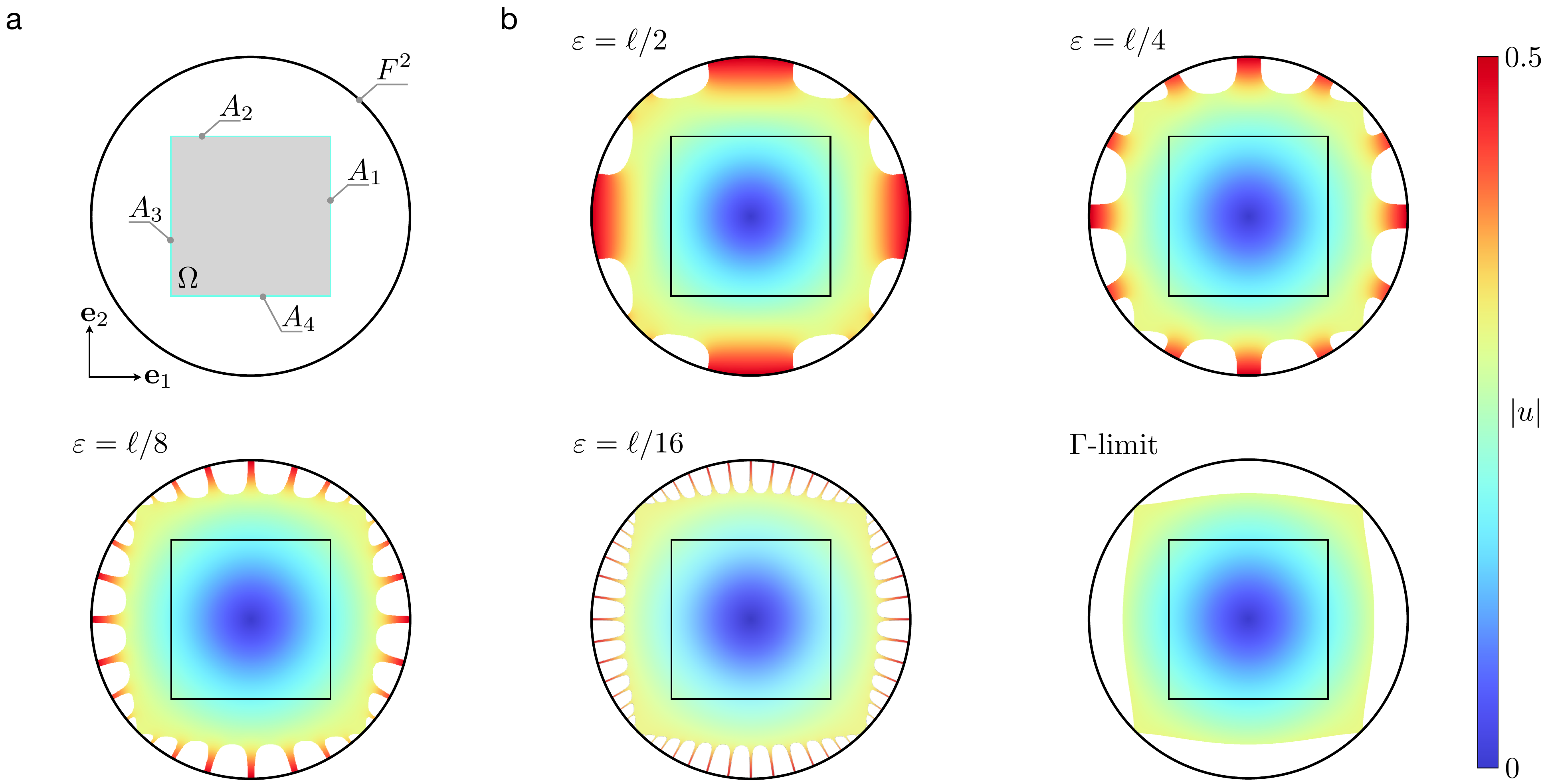}
    \caption{Numerical results for the case of a square domain $\Omega=(0, \ell)\times 0, \ell)$ of side length $\ell$ subjected to the constraint $F^2_{(x_1, x_2)}=\{(z_1, z_2): (z_1 + x_1 - x_{1,\textrm{c}})^2 + (z_2 + x_2 - x_{2,\textrm{c}})^2 = R^2\}$ at the perforations $A_1=\{\ell\}\times (0, \ell)$, $A_2 = (0, \ell)\times \{\ell\}$, $A_3=\{0\}\times (0, \ell)$, and $A_4=\{0\}\times(0, \ell)$. (a) Sketch of the boundary-value problem. (b) Deformed configurations for perforations with $\varepsilon = \{\ell/2, \ell/4, \ell/8, \ell/16\}$ and for the limit functional. The colorbar reports the norm of the displacement field. For all the results, $p = 3/2$, $\ell = 1$, $x_{1,\textrm{c}} = x_{2,\textrm{c}} = 1/2$, and $R = 1$.}
    \label{fig:ex_num_2}
\end{figure}

In both cases, numerical solutions to the boundary-value problems are computed through the finite element discretization of the governing equations. For this purpose, we exploit the functionalities of the software COMSOL Multiphysics~6.1. In particular, we discretize the displacement field by means of a suitable mesh of order-one Lagrange elements and seek numerical solutions via the Newton's solver.

In passing, we remark that an explicit formula for the $p$-capacity, $c$, appearing in the limit functional is not available for the case of boundary perforations in the form of segments of length $\delta$. Thus, its value has been numerically computed by independent finite element simulations, obtaining $c \approx 1.602$ for the value of $p = 3/2$. As later recalled, this is the choice for the numerical simulations concerning the first, the second, and the fourth problem. On the contrary, $c = \pi [2(2-p)/(p-1)]^{p-1}$ for the case of perforations in the interior of the reference set in the form of balls of diameter $\delta$, the form relevant to the third problem, see for instance \cite{Ma11}.

The numerical results for the first boundary-value problem are summarized in Fig.~\ref{fig:ex_num_1} as obtained for $p=3/2$, $\ell=1$, and $x_{1,\textrm{c}}=6/5$. 
Specifically, the problem is sketched in Fig.~\ref{fig:ex_num_1}a, whereas Fig.~\ref{fig:ex_num_1}b reports about the solutions obtained for $\varepsilon = \{1/2,1/4,1/8,1/16\}$ and for $\varepsilon \to 0$.
We observe that, as $\varepsilon$ decreases, the displacement field progressively approaches the $\Gamma$-limit, modulo on the discrete perforations where the constraint is enforced. This behaviour is expected in light of the mathematical result of Theorem~\ref{thm4}. Note that the $\Gamma$-limit corresponds to a homogeneous deformation, for the specific choice of strain energy density, and is such that the right vertical boundary of the reference set does not match the constraint of $F^1$. This is a characteristic feature of how the limit functional captures the averaged effect of the constraint on the perforation, as highlighted by the comparison with the numerical results for $\varepsilon = \{1/2, 1/4, 1/8,1/16\}$.

\begin{figure}[t]
    \centering
\includegraphics[width=1.0\linewidth]{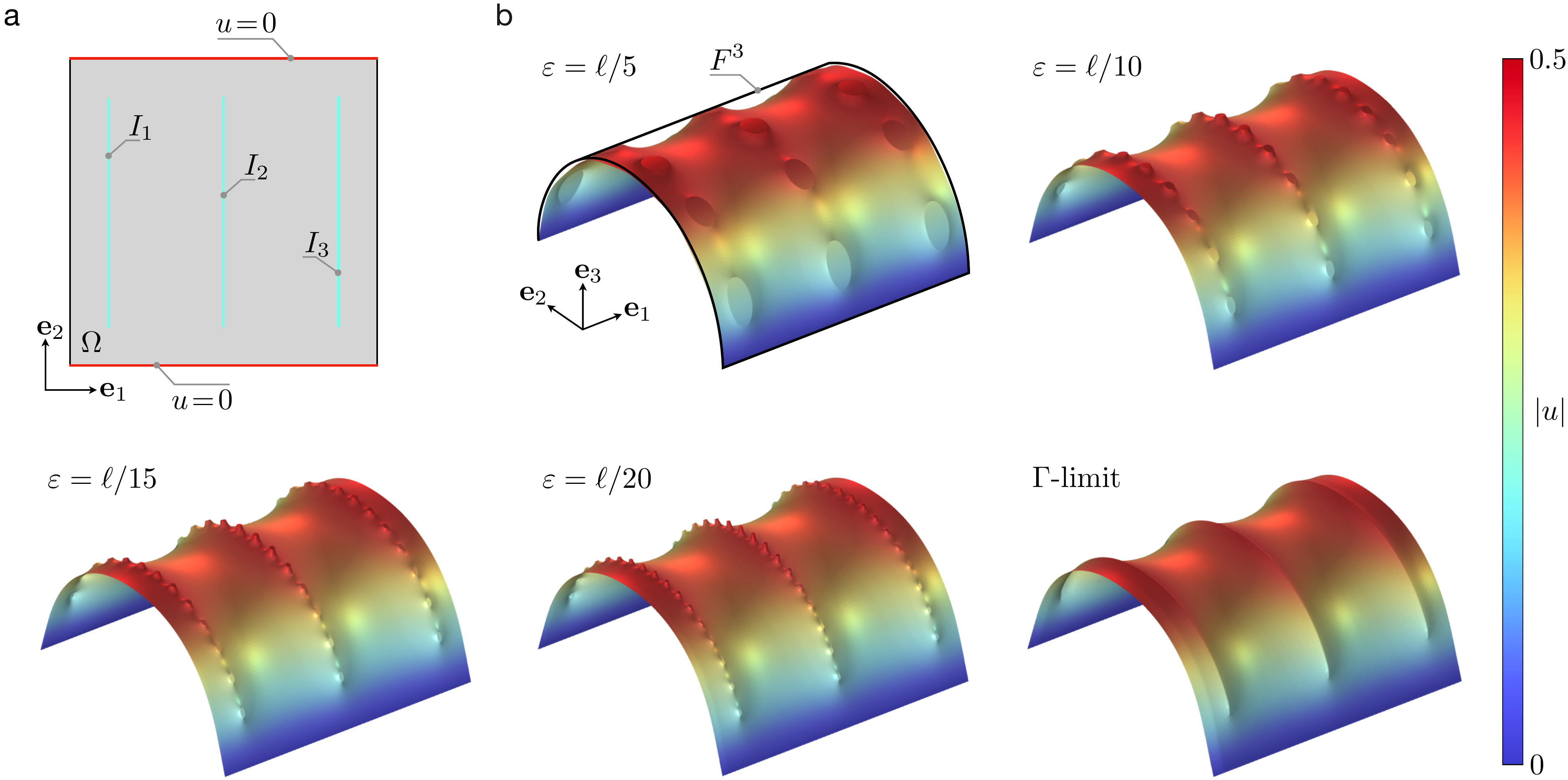}
    \caption{Numerical results for the case of a square domain $\Omega=(0, \ell)\times 0, \ell)$ of side length $\ell$ subjected to null displacement on $(0,\ell)\times\{0\}$ and on $(0,\ell)\times\{\ell\}$, and to the constraint $F^3_{(x_1, x_2)}=\{(z_1, z_2, z_3): (z_2 + x_2 - x_{2,\textrm{a}})^2 + (z_3)^2 = R^2\}$ on the one-dimensional perforations $I_1 = \{\ell/10\}\times (\ell/10,9\ell/10)$, $I_2 = \{\ell/2\}\times (\ell/10,9\ell/10)$, and $I_3 = \{9\ell/10\}\times (\ell/10,9\ell/10)$. (a) Sketch of the boundary-value problem. (b) Deformed configurations for perforations with $\varepsilon = \{\ell/5, \ell/10, \ell/15, \ell/20\}$ and for the limit functional. The colorbar reports the norm of the displacement field. For all the results, $p = 4/3$, $\ell = 1$, $x_{2,\textrm{a}} = 1/2$, and $R = 1/2$.}
    \label{fig:ex_num_3}
\end{figure}

\begin{figure}[ht]
    \centering
\includegraphics[width=1.0\linewidth]{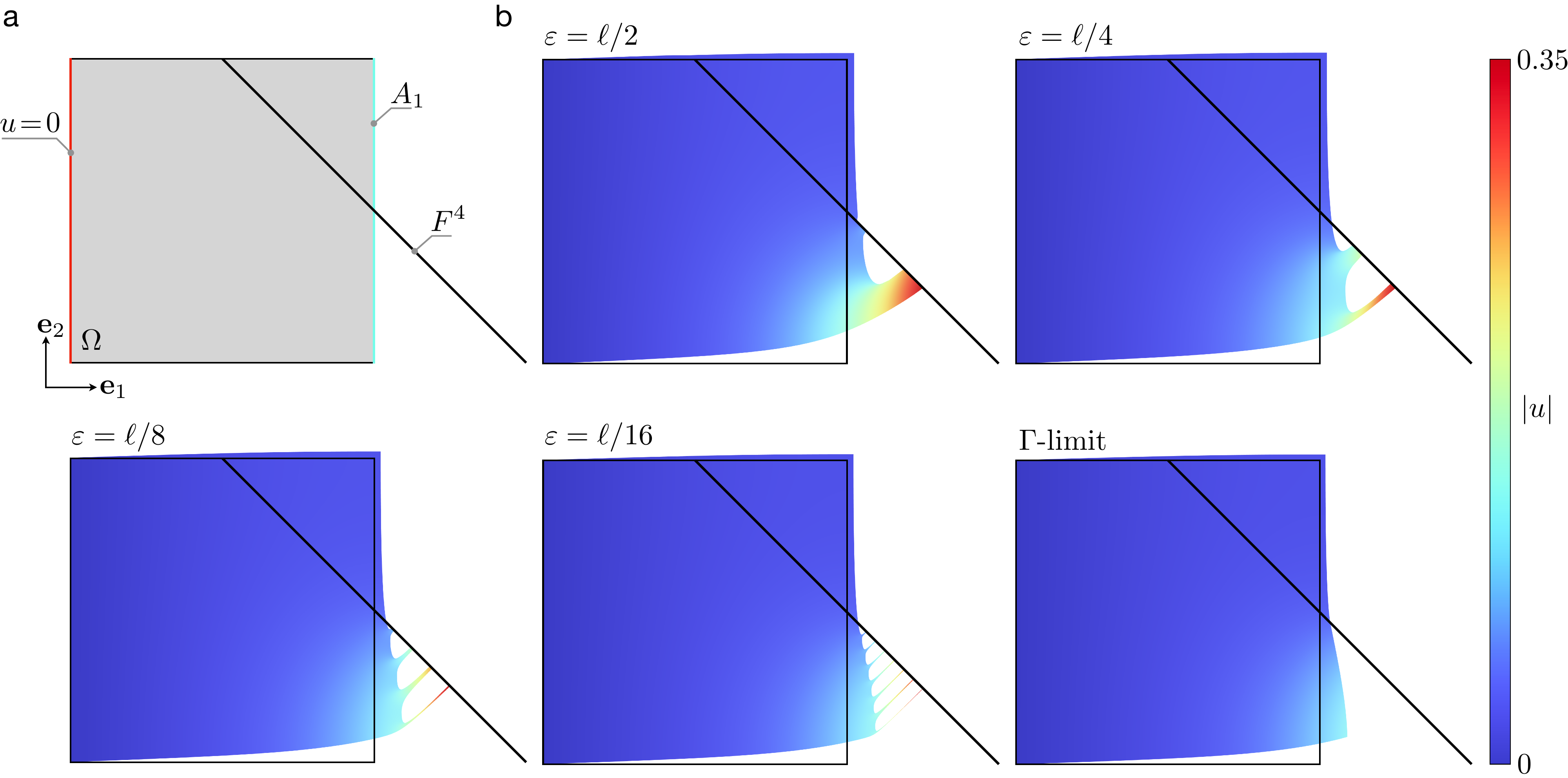}
    \caption{Numerical results for the case of a square domain $\Omega=(0, \ell)\times (0, \ell)$ of side length $\ell$ subjected to Dirichlet boundary conditions on $\{0\}\times (0, \ell)$ and to the constraint $F^4_{(x_1, x_2)}=\{(z_1, z_2): z_2 + x_2 + z_1 + x_1 - 3\ell/2 \ge 0\}$ on the perforation $A_1=\{\ell\}\times (0, \ell)$. (a) Sketch of the boundary-value problem. (b) Deformed configurations for perforations with $\varepsilon = \{\ell/2, \ell/4, \ell/8, \ell/16\}$ and for the limit functional. The colorbar reports the norm of the displacement field. For all the results, $p = 3/2$ and $\ell = 1$.}
    \label{fig:ex_num_4}
\end{figure}

We next report in Fig.~\ref{fig:ex_num_2} the numerical results for the second boundary-value problem as obtained for $p=3/2$, $\ell = 1$, $x_{1,\textrm{c}} = x_{2,\textrm{c}} = 1/2$, and $R = 1$. We recall that, in this case, perforations at the boundary of the reference set are constrained to lay on a circle of radius $R$ and center $(x_{1,\textrm{c}},x_{2,\textrm{c}})$, see Fig.~\ref{fig:ex_num_2}a. As for the previous case, the displacement field approaches the $\Gamma$-limit as the perforations' density increases. This is evident from the deformed shapes reported in Fig.~\ref{fig:ex_num_2}b. In this case the deformation is not homogeneous, but we further remark that in the $\Gamma$-limit the boundary of the reference set does not correspond to the constraint of $F^2$, as commented for the previous test case. 

Clearly, the two problems discussed above pertain the plane deformation of the reference configuration as due to constraints at its boundary. Hence, we turn our attention to the third boundary-value problem, such that one-dimensional perforations in the interior of the reference set determine its deformation in space, see Fig.~\ref{fig:ex_num_3}. Indeed, we recall that such perforations are constrained to lay on the cylinder of radius $R$ and axis the line of $x_2 = x_{2,\textrm{c}}$. The problem is sketched in Fig.~\ref{fig:ex_num_3}a and the relevant numerical results are shown in Fig.~\ref{fig:ex_num_3}b as obtained for $p = 4/3$, $\ell = 1$, $x_{2,\textrm{a}} = 1/2$, and $R = 1/2$. Specifically, the figure shows the results for $\varepsilon = \{1/5, 1/10, 1/15, 1/20\}$, to be compared with the limit case of $\varepsilon \to 0$. Also in this case, the $\Gamma$-limit well captures the effect of perforations of finite size as the parameter $\varepsilon$ decreases.

To conclude this section, we consider an interesting variation of the first boundary-value problem to further highlight the possibility to account for general constraints offered by the presented framework. In particular, we now assume that the boundary perforation is unilaterally constrained under tension to the line $x_2 + x_1 - 3\ell/2 = 0$, see Fig.~\ref{fig:ex_num_4}a. The numerical results for this test case are reported in Fig.~\ref{fig:ex_num_4}b as obtained for $\varepsilon = \{1/2, 1/4, 1/8, 1/16\}$ and for $p=3/2$ and $\ell = 1$. As for the previous cases, the $\Gamma$-limit well captures the averaged effect of the perforation as $\varepsilon$ decreases. To obtain the numerical results corresponding to the limit functional, the unit step function in the expression for $\overline{\varphi}$ has been approximated as $H(\tau) \approx (1 + \tanh(k\,\tau))/2$, with $k = 10^4$.

The result we have presented concern simple test cases and are limited to a specific choice of the strain energy density. Nevertheless, they provide evidence about the applicability of the proposed approach to the study of membranes and tensile structures subjected to diffuse constraints and allow to capture the average effect of such constraints on the equilibrium configuration of those structures.

\section{Conclusions and perspectives}

In this study, we have explored the behaviour of hyperelastic bodies, with energy function $\sigma$, subjected to kinematic constraints in many small regions. We have expressed our analysis through the computation of a $\Gamma$-limit, highlighting the overall effect of the perforation by the appearance of an extra integral term of a function $\varphi(x,u)$ depending on the $p$-recession function of $\sigma$, on the constraint at $x$, and on the value of the displacement $u$. The results that we obtain significantly improve previous asymptotic analyses, where mainly only a homogeneous Dirichlet  condition had been addressed in a scalar or convex case.
The numerical results, even though computed in simple situations, nicely show the homogenization effect of the perforation and how the $\Gamma$-limit problem simplifies the computations, still capturing the limit details of minimizers.

\begin{figure}[ht]
    \centering
\includegraphics[width=1.0\linewidth]{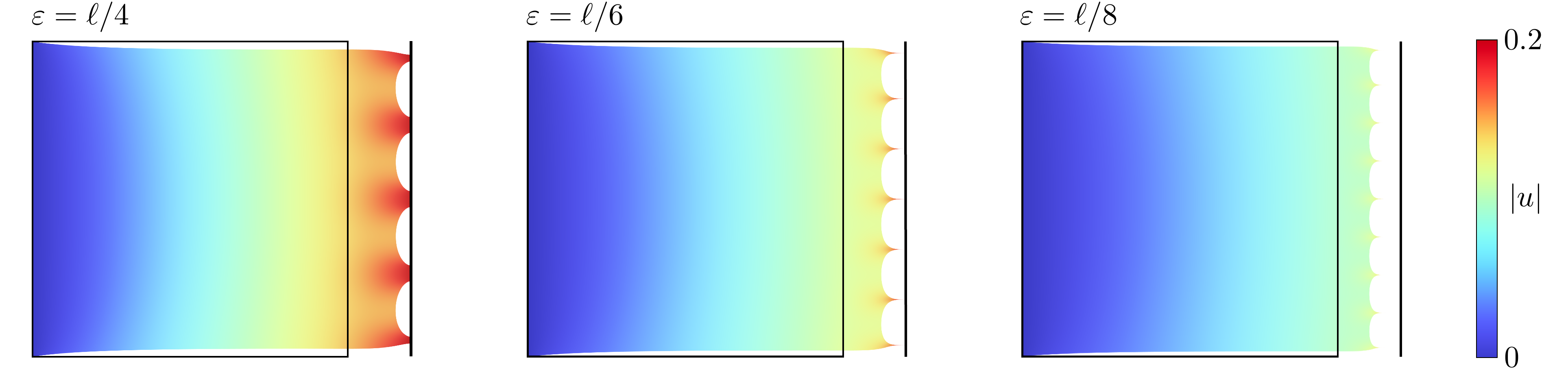}
    \caption{The first boundary-value problem is reexamined by assuming the compressible neo-Hookean material model. Deformed configurations for perforations with $\varepsilon = \{\ell/4, \ell/6, \ell/8\}$. The colorbar reports the norm of the displacement field. For all the results, $\mu = \lambda = 1$ and $\ell = 1$.}
    \label{fig:ex_num_5}
\end{figure}

In order to obtain our results, we have required some polynomial growth conditions on $\sigma$ that are commonly used in asymptotic problems for hyperelastic materials (see e.g.~\cite{BrDe98, Muller}) and that guarantee a good definition of $\varphi$. It is interesting to note, however, that the integrals appearing in the $\Gamma$-limit, both the elastic energy and the integral involving $\varphi$, are well defined and give a lower-semicontinuous functional, even if we require a polyconvex growth condition (see \cite{Ball,BrDe98}) and the existence of some $p$-recession function for $\sigma$. It remains an open problem whether such candidate limit energy is still a good approximation for this asymptotic problem in the polyconvex case. In this direction, we have performed a 
test resembling the first boundary-value problem of Section~\ref{sec:num_res}, but now assuming a compressible neo-Hookean material model. The strain energy density now takes the form
$$\sigma(\nabla y) = \frac{\mu}{2} (\|\nabla y\|^2 - 2 - \log J) + \frac{\lambda}{2}(\log J)^2 ,$$
where $J = \det(\nabla y)$ is assumed to be positive, while $\mu$ and $\lambda$ are material constants. In this case the energy does not satisfy a polynomial growth conditions due to the condition of $J>0$. However, we can conjecture that 
we can proceed as in the case
$d=p=2$, taking the observations at the end of Section \ref{sec:mat_results} into the account. In support of this, we report in Fig.~\ref{fig:ex_num_5} the results of numerical experiments carried out for $\varepsilon = \{\ell/4, \ell/6, \ell/8\}$ and assuming $\mu = \lambda = 1$ for the material constants. In setting the numerical simulations, we assumed the scaling of $\delta_\varepsilon = \exp(-\kappa/\varepsilon^2)$ with $\kappa = 1.9$.

We finally note that, for what regards the form of the constraints examined in this work, we have moved from a (relatively) simple condition modelled on $u=0$ or $m=1$ and $u\ge 0$ to an inclusion $u\in F_x$, with a large freedom on the closed set $F_x$. In order to pass to the limit, we have required some uniform continuity on the dependence $x\mapsto F_x$. We think that this is a technical assumptions, which can be relaxed in many ways. This will be useful in order to treat cases when the choice of the local form of $F_x$ may vary greatly, or is chosen randomly from a set of constraints, and we have a superposed homogenization effect.

\section*{Acknowledgements}

This work was supported by the Italian Ministry of University and Research (MUR) through the grants `Dipartimenti di Eccellenza 2023-2027 (Mathematics Area)' and PRIN 2022 n.\,2022NNTZNM `DISCOVER'. 
A.B. and S.V. are members of the `Gruppo Nazionale per l’Analisi Matematica, la Probabilità e le loro Applicazioni' and G.N. of the `Gruppo Nazionale di Fisica Matematica' of the `Istituto Nazionale di Alta Matematica' (INdAM).

\appendix
\renewcommand\appendixname{\!}

\section{Appendix: technical results}\label{app:lemmas}

    
In this Appendix we gather the  technical lemmas used in the convergence theorems, with a short proof indicating the main new arguments with respect to the known results.
\smallskip

The following result is a `truncation lemma', stating that, up to an arbitrary small error, the energies of a weakly converging sequence in $W^{1,p}$ are concentrated on sets where this sequence is equibounded. A precise statement is obtained by producing an equibounded sequence with almost the same limit and almost the same energy. This sequence is obtained from the original one by composition of bounded Lipschitz functions (note that a simple component-wise truncation may indeed increase the energies).

\begin{lm}\label{lemma: truncation} Let $\F_\varepsilon$ be as in Theorem {\rm\ref{thm1}}.
    Let $u_\varepsilon$ weakly converge to $u$   in $W^{1, p}(\Omega;\mathbb R^m)$. 
    Then for all $\eta>0$ there exist $R_\eta>0$, a subsequence $u_{\varepsilon_h}$ and a sequence $v_h\in W^{1, p}$ such that $\|v_h\|_{L^\infty}\leq R_\eta$,  $\|u_{\varepsilon_h}-v_h\|_{L^p}=o_\eta(1)$ as $\eta\to 0$, and $
     \liminf\limits_{\varepsilon\to 0} \mathcal F_\varepsilon(u_\varepsilon)\geq \liminf\limits_{h\to+\infty} \mathcal F_{\varepsilon_h}(v_h)-\eta$.

\end{lm}
\begin{proof} Let $s$, $\psi_j$, $M$ and $C$ be as in Condition 2 in the assumptions on $F_x$.
By a mollification argument, for $\varepsilon$ universally small one can produce $2C$-Lipschitz functions $\psi^\varepsilon_j(x, w)$ such that
\begin{equation}
    \psi^\varepsilon_j(x, w)\begin{cases}
        =z\qquad &\text{if }|w|\leq j,\\
        =s(x)\qquad &\text{if }|w|\geq 2j,\\
        \in F_{x^\varepsilon_i} &\text{for all }i\hbox{ and  }x\in K^\varepsilon_i,\hbox{ and } w\in F_{x^\varepsilon_i},
    \end{cases}
\end{equation}
with $\|\psi^\varepsilon_j-\psi_j\|_{L^\infty}=o_\varepsilon(1)$ as $\varepsilon\to 0$ uniformly in $j$. The proof follows along the lines of \cite[Lemma 3.5]{BrDeVi96} in order to produce $R_\eta>M$, a sequence $j_h$ such that $j_h+2C j_h\leq R_\eta$ and a subsequence $u_{\varepsilon_h}$ such that, setting $v_h\coloneqq \psi^\varepsilon_{j_h}(\cdot, u_{\varepsilon_h})$, we have $\liminf\limits_{\varepsilon\to 0} \mathcal F_\varepsilon(u_\varepsilon)\geq \liminf\limits_{h\to+\infty} \mathcal F_{\varepsilon_h}(v_h)-\eta$.
The other results in the statement follow from the properties of $\psi^\varepsilon_j$.
\end{proof}

The main tool in the treatment of periodic perforations is the following result. It states that sequences weakly converging in $W^{1,p}(\Omega;\mathbb R^m)$ can be modified without changing the limit and, up to introducing an arbitrarily small error depending on a parameter $N$, also their energy in such a way that they are constant on a collection of spherical hypersurfaces around the perforations of radii much larger than $\delta$ but still much smaller than $\varepsilon$, and chosen in a family of $N$ spherical hypersurfaces. This result has been used in \cite{AnBr02} for treating homogeneous Dirichlet conditions on perforations. It is a useful observation to note that the construction of the modified sequence does not involve the perforation, so that the lemma can be used for any constraint.

\begin{lm}\label{lemma: change bd}
    Let $u_\varepsilon$ converge weakly to $ u$ in $W^{1, p}(\Omega; \R^m)$ and let $\rho_\varepsilon$ be such that $0<\rho_\varepsilon<\varepsilon$ and both $\rho_\varepsilon=o(\varepsilon)$ and $\varepsilon^{1+\frac{p}{d}}=o(\rho_\varepsilon)$ hold. Then for all $N\geq 1$ there exist functions $i_\varepsilon\colon\mathbb Z^d\to \{1, \ldots, N\}$ and $\varv_\varepsilon\in W^{1, p}(\Omega;\R^m)$ such that $\varv_\varepsilon$ still weakly converges to $u$, 
$$    \Big|\mathcal \int_\Omega \sigma_\varepsilon(\nabla u_\varepsilon)\,\d x-\mathcal \int_\Omega \sigma_\varepsilon(\nabla\varv_\varepsilon)\,\d x\Big|=O\Big(\frac{1}{N}\Big),$$
if $C^\varepsilon_j=B(\varepsilon j, 2^{-2i_\varepsilon(j)}\rho_\varepsilon)\setminus B(\varepsilon j, 2^{-2(i_\varepsilon(j)+1)}\rho_\varepsilon)$ and $S_j^\varepsilon=\partial B(\varepsilon j, 2^{-2 i_\varepsilon(j)-1}\rho_\varepsilon)$, we have $\varv_\varepsilon=u_\varepsilon$ in $\Omega\setminus\bigcup_j C^\varepsilon_j$ and $\varv_\varepsilon\equiv\bar{\varv}^\varepsilon_j$ is constant on each $S^\varepsilon_j$.
Moreover, the piecewise-constant functions
$\sum_j \mathbbm{1}_{\varepsilon (j+(-\frac{1}{2}, \frac{1}{2})^d)}\bar{\varv}^\varepsilon_j$ converge to $u$ in $L^p(\Omega;\R^m)$.
Finally, if also $\sup_\varepsilon\|u_\varepsilon\|_{L^\infty}=S<+\infty$ then the same holds for $\varv_\varepsilon$.
\end{lm}
\begin{proof}
  The lemma is almost the same as \cite[Lemma 3.1, Lemma 4.3]{AnBr02}, with some minor changes in the notation. Note that the condition $\varepsilon^{1+\frac{p}{d}}\ll\rho_\varepsilon\ll\varepsilon$ allows to estimate the distance from their average of the functions $u_\varepsilon$ on the annuli $C^\varepsilon_j$ 
  after scaling to a reference annulus  $C_1\coloneqq B(0,\frac{1}{4})\setminus \overline{B}(0, \frac{1}{16})$.
  Indeed, by the scaling properties of the Poincar\'e inequality, if $C_\rho\coloneqq \rho C_1$, with $\rho\in (0, 1)$ then there exists $c=c(p, d)$ such that for all $u \in W^{1, p}((0, 1)^d;\mathbb R^m)$ it holds
    \begin{equation}
        \int_Q|u-u_\rho|^p\,\d x\leq \frac{c}{\rho^d}\int_Q \|\nabla u\|^p\,\d x,
    \end{equation}
    where $u_\rho=|C_\rho|^{-1}\int_{C_\rho}u\,\d x$.
\end{proof}

In order to understand the relaxation of the constraint in the limit, we study the asymptotic behaviour of the capacity-type problem \eqref{eq: cap er}. In the following we assume the relation $R_\varepsilon=\frac{\delta_\varepsilon}{\rho_\varepsilon}$ to hold, where $\delta_\varepsilon$ and $\rho_\varepsilon$ are to be specified. In this case, the problem we are tasked with studying is
\begin{equation}\label{eq: approx cap}
    \varphi_\varepsilon (x, u)\coloneqq \inf\left\{\int_{B(0, R_\varepsilon)} \sigma_\varepsilon(\nabla \varv)\,\d z,  \text{ where } \varv-u\in W^{1, p}_0(B(0, R_\varepsilon); \R^m), \varv\in F_x \hbox{ on } K\right\},
\end{equation}
 to be compared with $\varphi(x,u)$ 
which can then be regarded as a limit problem. 
\begin{lm}\label{lemma: pw}
    Assume that \eqref{quantitative} holds and let $\delta_\varepsilon=\varepsilon^{\frac{d}{d-p}}$.
    If $\rho_\varepsilon$ satisfies $\varepsilon^{1+\frac{p}{d}}\ll\rho_\varepsilon\ll \varepsilon^{\frac{d-\alpha p}{d-p}}
    $ (which is not an empty condition), then $\varphi_\varepsilon$ converges pointwise to $\varphi$.
\end{lm}
\begin{proof}
    We start by noting that $\frac{d-\alpha p}{d-p}< 1+\frac{p}{d}$. 
    We show that for $\rho_\varepsilon$ as in the statement $\varphi_\varepsilon$ converges to $ \varphi$ pointwise as $\varepsilon\to 0$.

    Let $\varepsilon>0$, $x\in \overline{\Omega}$, $w\in \mathbb R^m$ and  let $\varv$ be a minimizer for the problem defining $\varphi_\varepsilon (x, w)$ as in \eqref{eq: approx cap}. Let $\tilde{\varv}$ denote the extension to the all of $\mathbb R^d$ of $\varv$ which is constantly $w$ outside of $B(0, R_\varepsilon)$.
    Then 
    \begin{align}
        \varphi(x, u)& = \int_{\R^d} \sigma_0(\nabla \tilde{\varv})\d z=\int_{B(0, R_\varepsilon)} \sigma_0(\nabla \varv)\d z\\
        &\leq \int_{B(0, R_\varepsilon)} \sigma_\varepsilon(\nabla \varv)\d z+ C\varepsilon^\frac{d\alpha}{d-p}\int_{B(0, R_\varepsilon)}(1+\|\nabla \varv\|^p)\d z\\
        &\leq \varphi_\varepsilon(x, u)+C \varepsilon^\frac{d\alpha}{d-p} R_\varepsilon^d\Big(1+\varepsilon^\frac{dp}{d-p}c_1^{-2}\Big)+C \varepsilon^\frac{d\alpha}{d-p} \int_{B(0, R_\varepsilon)}\frac{1}{c_1} \sigma_\varepsilon(\nabla \varv)\d z\\
        &= \varphi_\varepsilon(x, u)+o_\varepsilon(1),
    \end{align}
    where the last equality follows from the fact that 
       $ \varepsilon^\frac{d\alpha}{d-p} R_\varepsilon^d=\varepsilon^{\frac{d}{d-p}( \alpha p-d)}\rho_\varepsilon^d\ll 1$.
    By taking the lower limit for $\varepsilon\to 0$, it follows that $\varphi\leq \liminf\limits_{\varepsilon\to 0} \varphi_\varepsilon$.
    
    The converse
inequality for the upper limit follows from a cut-off argument and a similar energy estimate.
\end{proof}

The pointwise convergence of $\varphi_\varepsilon$ can be improved to local uniform convergence as follows.

\begin{lm}\label{lemma: uniform}
    Let $\alpha, \rho_\varepsilon$ and $\delta_\varepsilon$ as above. Then $\varphi_\varepsilon$ converges to $\varphi$ uniformly on $\Omega\times B(0, L)$ for all $L>0$.
\end{lm}
\begin{proof}
    Recall that by the $p$-growth condition and the quasiconvexity of $\sigma$ we have 
    \begin{equation}\label{eq: Lipschitz}
        |\sigma_\varepsilon(A)- \sigma_\varepsilon(B)|\leq C \Big(\varepsilon^{\frac{p(d-1)}{d-p}}+\|A\|^{p-1}+\|B\|^{p-1}\Big)\,\|A-B\|,
    \end{equation}
    see \cite{BrDe98}.
    Let $\varv\in W^{1, p}(B(0, R_\varepsilon); \R^m)$ and let $\Psi\colon\R^m\to \R^m$ a bi-Lipschitz map.
    After some computations \eqref{eq: Lipschitz} yields
    \begin{align}
        \bigg|\int_{B(0, R_\varepsilon)}\sigma_\varepsilon(\nabla\varv)\d z-\int_{B(0, R_\varepsilon)}\sigma_\varepsilon\big(\nabla(\Psi\circ \varv)\big)\d z\bigg|&\leq C\varepsilon^\frac{d(p-1)}{d-p}R_\varepsilon^\frac{d(p-1)}{p}\|\nabla \Psi\|_{L^\infty}\|\nabla \Psi-I\|_{L^\infty} \bigg(\int_{B(0, R_\varepsilon)}\|\nabla \varv\|^p\d z\bigg)^\frac{1}{p}\\
        &\quad + C(\|\nabla \Psi\|_{L^\infty}^{p-1}+1)\|\nabla \Psi-I\|_{L^\infty}\int_{B(0, R_\varepsilon)}\|\nabla \varv\|^p\d z.
        \label{eq: fund Lip est}
    \end{align}
    
    With fixed $0<\eta<1$, $L>1$ and $x$, let $u_1, u_2\in B(0, L)$ be such that ${\rm dist}(u_1, F_x), {\rm dist}(u_2, F_x)>\eta$. Then there exists $\Psi\colon\R^m\to \R^m$ bi-Lipschitz such that $\Psi(F_x)=F_x$, $\Psi(u_2)=u_1$ and $\|\nabla \Psi-I\|_{L^\infty}\leq \frac{C}{\eta}|u_1-u_2|$.
    Applying \eqref{eq: fund Lip est} with $\varv$ a minimizer for \eqref{eq: approx cap} for $u_2$ and then exchanging the roles of $u_1$ and $u_2$, we get
    \begin{equation}\label{eq: Lip z}
        |\varphi_\varepsilon(x, u_1)-\varphi_\varepsilon(x, u_2)|\leq C\bigg(\frac{L^{2p+2}}{\eta^{p+2}}+o_\varepsilon(1)\bigg)\,|u_1-u_2|.
    \end{equation}
    Given $x_1, x_2$, we can test \eqref{eq: fund Lip est} on a minimizer $\varv$ for problem \eqref{eq: approx cap} and $\Psi=\Phi_{x_1, x_2}$. In this case we get
    \begin{equation}
        \varphi_\varepsilon(x_2, \Phi_{x_1, x_2}(u))\leq \varphi_\varepsilon(x_1, u)+C(L^{p+2}+o_\varepsilon(1))|x_1-x_2|.
    \end{equation}
    If $x_1, x_2$ are close enough (so that for instance  $\|\Phi_{x_1, x_2}-id\|, \|\Phi_{x_2, x_1}-id\|\leq \frac{\eta}{4}$ on $B(0, L)$) then this inequality can be coupled with \eqref{eq: Lip z} and yield
    \begin{equation}
        |\varphi_\varepsilon(x_1, u_1)-\varphi_\varepsilon(x_2, u_2)|\leq (C(L, \eta)+o_\varepsilon(1))(|x_1-x_2|+|u_1-u_2|)
    \end{equation}
    for $u_1, u_2$ satisfying $|u_1|, |u_2|\leq L$, ${\rm dist}(u_1, F_{x_1}), {\rm dist}(u_2, F_{x_2})>\eta$.
    This equi-continuity result, coupled with the crude estimate
    \begin{equation}
        |\varphi_\varepsilon(x_1, u_1)-\varphi(x_1, u_1)|\leq C\eta^p+o_\varepsilon(1)
    \end{equation}
    holding for $(x_1, u_1)$ such that ${\rm dist}(u_1, F_{x_1})\leq \eta$, allows to 
    conclude by an Ascoli--Arzelà argument.
\end{proof}



{\bf References}
\bibliographystyle{elsarticle-num} 
\bibliography{references}

\begin{thebibliography}{10}
\expandafter\ifx\csname url\endcsname\relax
  \def\url#1{\texttt{#1}}\fi
\expandafter\ifx\csname urlprefix\endcsname\relax\def\urlprefix{URL }\fi
\expandafter\ifx\csname href\endcsname\relax
  \def\href#1#2{#2} \def\path#1{#1}\fi

\bibitem{BK}
K.~Bhattacharya, A.~Braides, Thin films with many small cracks, R. Soc. Lond.
  Proc. Ser. A Math. Phys. Eng. Sci. 458 (2002) 823--840.
\newblock \href {http://dx.doi.org/10.1098/rspa.2001.0821}
  {\path{doi:10.1098/rspa.2001.0821}}.

\bibitem{gurtin_1981_a}
M.~E. Gurtin, An {I}ntroduction to {C}ontinuum {M}echanics, Academic Press,
  1981.

\bibitem{gurtin_1981_b}
M.~E. Gurtin, Topics in {F}inite {E}lasticity, Society for Industrial and
  Applied Mathematics, 1981.

\bibitem{March-Krus}
V.~A. Marchenko, E.~Y. Khruslov, Boundary {V}alue {P}roblems in {D}omains with
  a {F}ine-Grained {B}oundary (Russian), Izdat. ``Naukova Dumka'', Kiev, 1974.

\bibitem{Cio-Mur}
D.~Cioranescu, F.~Murat, Un terme \'etrange venu d'ailleurs, in: Nonlinear
  {P}artial {D}ifferential {E}quations and their {A}pplications. {C}oll\`ege de
  {F}rance {S}eminar, {V}ol. {II} ({P}aris, 1979/1980), Vol.~60 of Res. Notes
  in Math., Pitman, Boston, Mass.-London, 1982, pp. 98--138, 389--390.

\bibitem{GCB}
A.~Braides, {$\Gamma$}-convergence for {B}eginners, Vol.~22 of Oxford Lecture
  Series in Mathematics and its Applications, Oxford University Press, Oxford,
  2002.
\newblock \href {http://dx.doi.org/10.1093/acprof:oso/9780198507840.001.0001}
  {\path{doi:10.1093/acprof:oso/9780198507840.001.0001}}.

\bibitem{Bha-Ja}
K.~Bhattacharya, R.~D. James, The material is the machine, Science 307~(5706)
  (2005) 53--54.
\newblock \href {http://dx.doi.org/10.1126/science.1100892}
  {\path{doi:10.1126/science.1100892}}.

\bibitem{ortigosa_2021}
R.~Ortigosa, J.~Martinez~Frutos, C.~Mora~Corral, P.~Pedregal, F.~Periago,
  Optimal control of soft materials using a {Hausdorff} distance functional,
  SIAM Journal on Control and Optimization 59~(1) (2021) 393--416.
\newblock \href {http://dx.doi.org/10.1137/19M1307299}
  {\path{doi:10.1137/19M1307299}}.

\bibitem{andrini_2022}
D.~Andrini, G.~Noselli, A.~Lucantonio, Optimal design of planar shapes with
  active materials, Proceedings of the Royal Society A: Mathematical, Physical
  and Engineering Sciences 478~(2266) (2022) 20220256.
\newblock \href {http://dx.doi.org/10.1098/rspa.2022.0256}
  {\path{doi:10.1098/rspa.2022.0256}}.

\bibitem{mina_2018}
M.~Konakovi\'{c}-Lukovi\'{c}, J.~Panetta, K.~Crane, M.~Pauly, Rapid deployment
  of curved surfaces via programmable auxetics, ACM Transactions on Graphics
  37~(4).
\newblock \href {http://dx.doi.org/10.1145/3197517.3201373}
  {\path{doi:10.1145/3197517.3201373}}.

\bibitem{andrini_2024}
D.~Andrini, M.~Magri, P.~Ciarletta, Optimal surface clothing with elastic nets,
  Journal of the Mechanics and Physics of Solids 188 (2024) 105684.
\newblock \href {http://dx.doi.org/10.1016/j.jmps.2024.105684}
  {\path{doi:10.1016/j.jmps.2024.105684}}.

\bibitem{AnBr02}
N.~Ansini, A.~Braides, Asymptotic analysis of periodically-perforated nonlinear
  media, Journal de Mathématiques Pures et Appliquées 81~(5) (2002) 439--451.
\newblock \href {http://dx.doi.org/10.1016/S0021-7824(01)01226-0}
  {\path{doi:10.1016/S0021-7824(01)01226-0}}.

\bibitem{Bhattacharya-book}
K.~Bhattacharya, Microstructure of {M}artensite, Oxford Series on Materials
  Modelling, Oxford University Press, Oxford, 2003.

\bibitem{BrDe98}
A.~Braides, A.~Defranceschi, Homogenization of {M}ultiple {I}ntegrals, Vol.~12
  of Oxford Lecture Series in Mathematics and its Applications, The Clarendon
  Press, Oxford University Press, New York, 1998.

\bibitem{sig}
L.~Sigalotti, Asymptotic analysis of periodically-perforated nonlinear media at
  the critical exponent, Commun. Contemp. Math. 11~(6) (2009) 1009--1033.
\newblock \href {http://dx.doi.org/10.1142/S0219199709003648}
  {\path{doi:10.1142/S0219199709003648}}.

\bibitem{Br-Sig}
A.~Braides, L.~Sigalotti, Asymptotic analysis of periodically-perforated
  nonlinear media at and close to the critical exponent, C. R. Math. Acad. Sci.
  Paris 346~(5-6) (2008) 363--367.
\newblock \href {http://dx.doi.org/10.1016/j.crma.2008.01.010}
  {\path{doi:10.1016/j.crma.2008.01.010}}.

\bibitem{BT}
A.~Braides, L.~Truskinovsky, Asymptotic expansions by {$\Gamma$}-convergence,
  Contin. Mech. Thermodyn. 20~(1) (2008) 21--62.
\newblock \href {http://dx.doi.org/10.1007/s00161-008-0072-2}
  {\path{doi:10.1007/s00161-008-0072-2}}.

\bibitem{Ma11}
V.~Maz'ya, Sobolev {Spaces}: with {Applications} to {Elliptic} {Partial}
  {Differential} {Equations}, Vol. 342 of Grundlehren der mathematischen
  {Wissenschaften}, Springer, Berlin, Heidelberg, 2011.
\newblock \href {http://dx.doi.org/10.1007/978-3-642-15564-2}
  {\path{doi:10.1007/978-3-642-15564-2}}.

\bibitem{Muller}
S.~M\"uller, Variational models for microstructure and phase transitions, in:
  Calculus of variations and geometric evolution problems ({C}etraro, 1996),
  Vol. 1713 of Lecture Notes in Math., Springer, Berlin, 1999, pp. 85--210.
\newblock \href {http://dx.doi.org/10.1007/BFb0092670}
  {\path{doi:10.1007/BFb0092670}}.

\bibitem{Ball}
J.~M. Ball, Convexity conditions and existence theorems in nonlinear
  elasticity, Arch. Rational Mech. Anal. 63~(4) (1976/77) 337--403.
\newblock \href {http://dx.doi.org/10.1007/BF00279992}
  {\path{doi:10.1007/BF00279992}}.

\bibitem{BrDeVi96}
A.~Braides, A.~Defranceschi, E.~Vitali, Homogenization of free discontinuity
  problems, Archive for Rational Mechanics and Analysis 135~(4) (1996)
  297--356.
\newblock \href {http://dx.doi.org/10.1007/BF02198476}
  {\path{doi:10.1007/BF02198476}}.

\end{thebibliography}

\end{document}